\documentclass[11pt,reqno]{amsart}

\usepackage[T1]{fontenc}
\usepackage[utf8]{inputenc}
\usepackage{amsmath}
\usepackage{amssymb}
\usepackage{amsthm}
\usepackage{array}
\usepackage{graphicx}

\usepackage{booktabs}
\usepackage{needspace}
\usepackage[margin=1.1in]{geometry}
\usepackage[hidelinks]{hyperref}
\hypersetup{pdfauthor={Enkai Zhang},pdftitle={Density regions, integer certificates and packing colorings of distance graphs}}

\theoremstyle{plain}
\newtheorem{theorem}{Theorem}
\newtheorem{lemma}[theorem]{Lemma}

\newcommand{\D}[1]{D(1,#1)}
\newcommand{\chirho}{\chi_{\rho}}

\theoremstyle{plain}
\newtheorem{Stheorem}{Theorem}
\theoremstyle{plain}
\newtheorem{Slemma}[Stheorem]{Lemma}
\providecommand{\MainPrefix}{}
\providecommand{\SuppPrefix}{}
\hypersetup{hypertexnames=false}
\begin{document}
\hypertarget{article-start}{}
\renewcommand{\SuppPrefix}{Supplement }

\title[Density regions for distance graphs]{Density regions, integer certificates and packing colorings\
of distance graphs}

\author{Enkai Zhang}
\address{University of Toronto Scarborough, Toronto, Ontario, Canada}
\email{\href{mailto:ek.zhang@mail.utoronto.ca}{ek.zhang@mail.utoronto.ca}}


\begin{abstract}
We prove that the packing chromatic number of the integer distance graph
$D(1,6)$ lies between 17 and 20, that of $D(1,8)$ between 18 and 22,
and that $D(1,9)$ admits a packing 17-coloring. The lower bounds 17
for $D(1,6)$ and 18 for $D(1,8)$ have computer-assisted proofs.
Integer potentials checked on finite graphs bound the joint occupancy of low
colors. Spacing bounds control the remaining colors.
Explicit periodic words give all three upper bounds. For $D(1,6)$ the first seven colors have maximum
combined density $211/252$, and approaching this maximum forces their
individual frequencies. At maximum low-five density, some pairs of
higher-color frequencies occur only in nonperiodic colorings. We give
the growth rate of the low-color deficit needed to realize them.
For a sufficiently large additional color $i$, a fixed finite graph
in each residue modulo 14 determines the joint density region. For
$i=36,50,64,\ldots$ this region has an explicit seven-vertex formula.
The formula fails at $i=22$, so 36 is the least starting index for
the entire tail in that residue class.
\end{abstract}

\maketitle

\section{Color densities and the main results}
For an integer $t\ge2$, the distance graph $D(1,t)$ has vertex set $\mathbb Z$;
two vertices are adjacent when their difference in absolute value is $1$ or
$t$. Distances are shortest-path distances in this infinite graph. A packing
$k$-coloring assigns a color in $\{1,\ldots,k\}$ to every vertex so that two
distinct vertices of color $i$ have distance greater than $i$. The least such
$k$ is $\chi_\rho(D(1,t))$.

A partial packing coloring may leave vertices uncolored, denoted by $0$.
Only positive colors are subject to separation constraints. For a finite
integer interval $I$, let $n_i(I)$ count its vertices of color $i$, and put
\[
 M_j(I)=\sum_{i=1}^j n_i(I).
\]
We ask which color frequencies can be achieved simultaneously. In particular,
a layout that maximizes $M_j(I)/|I|$ need not leave the best possible positions
for additional colors. We study this interaction and its consequences for complete colorings.

The ordinary density of color $i$, when it exists, is
$\rho_i=\lim_{N\to\infty}n_i([0,N)\cap\mathbb Z)/N$. Unless a theorem specifies ordinary prefixes, a simultaneous
limiting vector is taken along one common sequence of intervals $I_N$ with
$|I_N|\to\infty$. Different coordinates are not allowed to choose different
subsequences. A coloring has uniformly bounded discrepancy at density $\rho_i$
for color $i$ if there is a constant $C$ such that
\[
 \bigl|n_i(I)-\rho_i\,|I|\bigr|\le C
\]
for every finite interval $I$. For a set of positions $S\subseteq\mathbb Z$, its upper Banach density is
\[
 \limsup_{N\to\infty}\ \sup_{a\in\mathbb Z}\frac{|S\cap[a,a+N)|}{N}.
\]
A periodic coloring repeats a finite block. Its density exists, but conclusions
about periodic realizations and conclusions about arbitrary limiting vectors
are kept separate below.

\paragraph{Main results.}
The chromatic bounds are
\[
 17\le\chi_\rho(D(1,6))\le20,\qquad
 18\le\chi_\rho(D(1,8))\le22,\qquad \chi_\rho(D(1,9))\le17.
\]
For the two lower bounds we add a joint occupancy bound for low colors
to spacing bounds for the remaining colors. The author developed both
the certificate generator and a second implementation, the verifier,
with AI assistance. The verifier reconstructs the transitions from the
packing constraints and checks each integer inequality. The three upper
bounds use periods 2520, 720 and 360.
The density analysis on $D(1,6)$ gives the following results.
\begin{enumerate}
\item \emph{Exact regions and forced frequencies.}
The feasible pairs consisting of the density of colors $1$--$4$ and
the density of color $5$ form the polygon of
Theorem~\ref{M-thm:global-density-region}. We also prove the unrestricted
seven-color maximum $211/252$ and determine the individual frequencies
whenever this maximum is approached (Theorem~\ref{M-thm:cr-seven}).
These results use both upper certificates and attaining constructions.
\item \emph{Optimal components need not support the same high colors.}
At the maximum low-five density $65/84$, two types of component give
different high-color capacities. Their ordinary limiting region is a
convex hull, while their periodically attainable points occupy a smaller
union of rectangles (Theorem~\ref{M-thm:infinite-pair-regions}).
Theorem~\ref{M-thm:component-deficit-rate} determines the asymptotic
cost of mixing the components when a target lies outside that union.
\item \emph{A fixed structure governs sufficiently large color indices.}
Theorem~\ref{M-thm:eventual-high-regions} describes all eventual low-five/high
regions using finitely many templates in each residue modulo $14$.
Theorem~\ref{M-thm:cr-residue8} makes this description explicit for
$i=36+14t$, $t\ge0$. Its seven-vertex formula fails at $i=22$.
\end{enumerate}

Sections~\ref{M-sec:certificate-principle}--\ref{M-sec:all-high} introduce the
certificate method and the optimal low-color structures needed in these
proofs. Sections~\ref{M-sec:density-regions}--\ref{M-sec:eventual-high-regions}
develop the regions and their realizations.
Section~\ref{M-sec:constr} and Supplement Sections S2--S3 give the
chromatic constructions and exclusions, respectively.

\paragraph{Related work.}
The use of joint low-color density bounds followed by high-color spacing
estimates goes back to Ekstein, Holub and Lidick\'y \cite{M-EHL12}. Finite-state and
cycle-mean methods for density problems in distance graphs have also been
used for independence ratios \cite{M-CGHRS16}. The vector-valued formulation is related
to rotation sets for subshifts of finite type: Ziemian \cite{M-Ziemian95} proves convex-hull
and periodic-approximation results under the stated transitivity hypotheses.
In the terminology of ergodic optimization \cite{M-Jen19}, a nonnegative
linear combination of the frequencies is a locally constant observable on
the space of legal paths. Its maximal average is a maximum cycle mean.
On the face where low-color occupancy is maximal, distinct optimal
components can have different high-color capacities; a periodic path
in that face stays in one component.

For packing-coloring background and complexity, see \cite{M-BKR07,M-EHT14,M-FG10,M-FKL09,M-GHHHR08,M-Survey20}. The square-grid
case was settled by SAT methods \cite{M-MRCM17,M-SH23}; \cite{M-BFHJM24,M-EF25} concern related $S$-packing
colorings, and \cite{M-HZZ25} gives further SAT-based grid bounds. For the three distance
graphs considered here, the earlier bounds are those recorded in \cite{M-EHL12,M-Togni14,M-SV15,M-SV16,M-Survey20}.
In particular, the corrigendum \cite{M-SV16} changes the printed direction of \cite[Theorem 1(i)]{M-SV15} to $\chi_\rho(D(1,6))\ge16$. We therefore use $16$ as the
previous lower bound and $23$ as the previous upper bound, rather than the
inconsistent upper-bound entry reproduced in \cite{M-Survey20}.

The earlier bounds for the other two distance graphs are
\cite{M-EHL12,M-Togni14,M-Survey20}
\[
 15\le\chi_\rho(D(1,8))\le25,\qquad 13\le\chi_\rho(D(1,9))\le18.
\]
The lower bound $13$ is due to Ekstein, Holub and Lidick\'y
\cite[Lemma~6]{M-EHL12}. The complete coloring words and the
chromatic lower-bound proofs are in the Technical Supplement.

\section{The metric and a finite certificate principle}\label{M-sec:certificate-principle}

\begin{lemma}[metric; {\cite[Lemma~1]{M-Togni14}}]\label{M-lem:metric}
If $\delta \ge 0$, $\delta = qt + r$ and $0 \le r < t$, then
$d_{\D{t}}(0,\delta) = q + \min(r,\,t+1-r)$.
\end{lemma}

\begin{proof}
The proof is included for self-containedness. A path can be rearranged into
signed unit steps and signed $t$-steps, so its
minimum length is $\min_b\,(|b| + |\delta - bt|)$ over integers $b$. For
$0 \le b \le q$ this expression decreases as $b$ increases; for $b \ge q+1$ it
increases. Negative $b$ cannot improve the value at $b = 0$. The two remaining
candidates $b = q$ and $b = q+1$ give the formula, and both are realized by
paths.
\end{proof}

For a finite set of selected colors $C$ put
\[
F_i = \{\,\delta > 0 \;:\; \delta = a + tb \text{ for integers } a,b
\text{ with } |a| + |b| \le i \,\}.
\]
Equivalently $F_i$ consists of the positive displacements at distance at most
$i$. It is finite and its largest element is $ti$. At a cut immediately before
the next position, record the ages of all occurrences of each selected color
among the preceding $ti$ positions; age $1$ denotes the immediately preceding
position. These tuples form history states.

From a state, append either a vacancy or one selected color. Appending $i$ is
allowed exactly when its age list is disjoint from $F_i$. Increase each old age
by one, delete ages greater than their respective horizons, and insert age $1$
for the appended color. Enumerating this rule from the empty state produces
every reachable state and every legal transition. Multiple edges, if present,
are checked as separate transitions. Assign a positive integer reward $w_i$ to
an occurrence of $i$, and reward zero to a vacancy.

\begin{lemma}[uniform certificate]\label{M-lem:cert}
Suppose integers $h(v)$, $p \ge 0$ and $q > 0$ satisfy
$h(v) - h(u) \ge q\,w(u,v) - p$ on every legal transition. If $h_0$ is the
empty-state potential and $H$ the maximum potential, then every interval $I$ of
$N$ positions in every packing coloring satisfies
\begin{equation}\label{M-eq:cert}
\sum_{i \in C} w_i\, n_i(I) \;\le\; \frac{p}{q}\,N + \frac{H - h_0}{q}.
\end{equation}
\end{lemma}

\begin{proof}
Replace unselected colors by vacancies and remove occurrences outside $I$. This
only removes restrictions, so the word on $I$ traces a legal path from the
empty state, by induction on its length. Sum the transition inequalities along
this path. The potentials telescope, and the final potential is at most $H$.
This proves \eqref{M-eq:cert}, uniformly in $I$ and $N$.
\end{proof}

This argument counts all colors on one common interval. A bound on the
$\limsup$ of the joint weighted count need not bound the weighted sum of the
individual $\limsup$s. Separately established single-color upper bounds can be
added, but that generally loses the joint restriction used here.

The \emph{mean reward} of a directed cycle is its total edge reward divided
by its number of edges. A \emph{potential} is a function on the graph
vertices satisfying the edge inequalities in Lemma~\ref{M-lem:cert}.
That lemma gives one direction of the relation between potentials and
maximum cycle means. Give every transition the length $q\,w(u,v)-p$. A potential with
$h(v)-h(u)\ge q\,w(u,v)-p$ exists if and only if no cycle has positive total length, that is,
if and only if every cycle of the automaton has mean reward at most $p/q$ (this is the
feasibility criterion for systems of difference constraints, and the maximum mean reward of a
cycle can be computed by Karp's algorithm~\cite{M-Karp78}); when it exists it can be taken to be
the longest-path distance from the empty state, which is an integer because the lengths are
integers. Each $p/q$ below is the exact maximum cycle mean of its automaton, computed in this way,
so the potentials are sharp: the attaining periodic partial colorings in Table~\ref{M-tab:scalar-certificates} are cycles of
mean reward exactly $p/q$.

\Needspace{17\baselineskip}
\paragraph{A small example.}
Take $D(1,2)$ and allow only colors $1,2$, together with vacancies. Their
forbidden positive displacements are $F_1=\{1,2\}$ and
$F_2=\{1,2,3,4\}$. The history graph has 13 reachable states and 21 edges.
Give reward one to either positive color and zero to a vacancy. A potential
with range two satisfies
\[
 h(v)-h(u)\ge2w(u,v)-1
\]
on every edge. Lemma~\ref{M-lem:cert} therefore gives
\[
 n_1(I)+n_2(I)\le\frac{|I|}{2}+1.
\]
The repeated word $100102$ is legal and has combined density $1/2$, so the
asymptotic bound is attained. Separately, colors 1 and 2 can have densities
$1/3$ and $1/5$; their sum is larger than $1/2$. The joint bound detects a
constraint lost by optimizing the colors one at a time. The complete 13-state
potential is printed in Supplement~\ref*{S-supp:worked-example}.

\subsection{A smaller recognizer with the same legal words}\label{M-sec:future-model}
For finite forbidden positive-displacement sets $F_c$ and a vacant
symbol $0$, let $B_c$ be the offsets at which a future occurrence of
color $c$ is forbidden by the current past. Appending $c>0$ is legal
exactly when $0\notin B_c$. Shift every set down by one and discard
negative offsets; if the appended symbol is $c$, additionally insert
$F_c-1$ in its set. Appending $0$ only shifts the sets. Start from
empty sets and retain all reachable tuples. We call this automaton the
future-constraint graph and its states forbidden-mask states.

Induction shows that this recognizes exactly the finite legal partial
colorings. It is also the minimal productive deterministic recognizer:
if two tuples differ at offset $t$ in coordinate $c$, the suffix
$0^tc$ distinguishes them. When some $F_c$ is nonempty, adding the
reachable rejecting sink gives the minimal total deterministic automaton.
With no forbidden displacements there is no rejecting sink. Density
potentials concern only legal productive edges.

Potentials transfer exactly at the level of the bound. If $\pi$ is the
label-preserving history quotient and $H$ is a valid old potential, set
\[
 \overline H(u)=\max_{\pi(x)=u}H(x).
\]
Histories in one fiber permit exactly the same next labels. For a quotient edge $u\mathrel{\smash{\xrightarrow{c}}}v$, choose
a maximizing history $x$ in the source fiber. Its $c$-successor $y$
lies in the target fiber, so
\[
 qw(c)\le p+H(y)-H(x)\le p+\overline H(v)-\overline H(u).
\]
The range does not increase. The same argument applies separately at
each fixed pair of high-color ages. A fiber minimum would not justify
this inequality. This preserves the density bound and its boundary range. It does not
identify the original graph's critical vertices or recover its individual
potential values.

\paragraph{State counts used in the certificates.}
The history graph for colors $1$--$4$ in $D(1,6)$ has $800532$ reachable
states. Its future-mask quotient has $253874$. To keep the location of
color 5, the one-age certificate takes all $253874\cdot31=7870094$
pairs of a low mask and a capped age. A breadth-first traversal from the
empty past using the five color masks instead gives $6523612$ productive
states and $17796241$ legal labelled edges. This is the minimal
deterministic recognizer on legal paths by the distinguishing suffix above.
For colors 5, 6 and 7 the certificate uses all
$253874\cdot31\cdot37\cdot43=12521319554$ formal tuples. Every legal
coloring traces a path in the corresponding product, so a potential on
all its edges bounds every such coloring. Equality constructions are
checked as periodic colorings. Critical-component sizes below refer to
the graph used to define them.

\subsection{Certificate data}
The generator uses binary histories. The verifier uses occurrence-age
lists to regenerate every transition and checks the five certificates
with integer arithmetic.

\begin{table}[h]
\centering
\footnotesize
\setlength{\tabcolsep}{4pt}
\begin{tabular}{cclcrrcrrr}
\toprule
Cert. & $t$ & Colors & Rewards & States & Transitions & $p/q$ & $h_0$ & $H$ & Period \\
\midrule
E & 6 & $1,2,3,4$ & $1,1,1,1$ & 800532 & 1755623 & $83/114$ & 0 & 322 & 114 \\
A & 8 & $1,2,4$   & $208,227,227$ & 592902 & 1114749 & $137$ & 0 & 560 & 86 \\
B & 8 & $1,5$     & $4,5$ & 6829 & 10270 & $2$ & 0 & 8 & 41 \\
C & 8 & $1,6$     & $1,16$ & 7269 & 10878 & $1$ & 0 & 15 & 28 \\
D & 8 & $3,5$     & $151,147$ & 24365 & 31145 & $3445/144$ & 0 & 50077 & 144 \\
\bottomrule
\end{tabular}
\caption{The five certificates.}\label{M-tab:scalar-certificates}
\end{table}

Each certificate includes a periodic partial coloring whose mean reward is
$p/q$. Every forbidden displacement is checked modulo its period, including
displacements that return to the same residue in a different period copy. Thus
the bounds on mean weighted occupancy are sharp. Attainment is not needed for
Theorem~\ref{M-thm:main}, which uses only \eqref{M-eq:cert}. For an arbitrary
partial coloring the precise upper-density formulation is the $\limsup$ of the
\emph{joint weighted count} on $[0,N)$ divided by $N$, not a sum of separate
$\limsup$s.

\subsection{Realizing density vectors}\label{M-sec:realization}
\begin{lemma}[Concatenation and thinning]\label{M-lem:density-realization}
Consider finitely many partial-coloring words with finite forbidden
displacement sets.
\begin{enumerate}
\item Suppose there is a legal history state $v$ such that both words label
closed walks from $v$ back to $v$. Then every convex
combination of their density vectors has a realization with uniformly
bounded interval discrepancy. Rational combinations have periodic
realizations.
\item For arbitrary periodic words, every convex combination of their
density vectors has a realization with ordinary density limits.
\item If a realization has ordinary limits, or uniformly bounded
discrepancy, independently thinning each color's occurrence sequence
preserves the stated density property at each smaller nonnegative target. A rational
target below a periodic vector also has a periodic realization.
\end{enumerate}
\end{lemma}
\begin{proof}
For (1), let the word lengths be $L_0,L_1$ and their count vectors
$v_0,v_1$. A two-sided balanced binary sequence with block frequency
$p$ has $pk+O(1)$ blocks of type 1 in any $k$ consecutive blocks.
Concatenation is legal by the shared history. Its density is
\[
 \frac{(1-p)v_0+pv_1}{(1-p)L_0+pL_1}.
\]
As $p$ varies this traces the segment between the two original
density vectors. Counts and lengths have bounded error on whole
blocks; an arbitrary interval adds at most two partial blocks.
For rational $p$, repeat a finite block sequence with that frequency.
This also realizes any rational combination of the two vectors.

For (2), separate blocks by a fixed number of vacancies greater than
every forbidden displacement. At stage $j$, take repetitions of each
periodic word with total length $j+O(1)$ in the required proportions.
Rounding and buffers contribute $O(1)$ per stage. After $k$ stages
the length is $k^2/2+O(k)$, while both the accumulated error and the
last incomplete stage have size $O(k)$. Hence ordinary densities
converge. The same construction on the two half-lines gives a
two-sided coloring.

For (3), retain a balanced fraction of the occurrences of each color.
On any consecutive set of occurrences the retained count differs
from that fraction by $O(1)$. Deleting colors cannot create a
conflict, so the stated density property is preserved. For rational
periodic targets, repeat enough periods to make every desired count
integral, select those occurrences, and repeat the resulting word.
\end{proof}

Buffers in (2) generally give unbounded accumulated low-color loss.
They do not imply exact periodic attainment on an optimal face.
Where a theorem requires the entire low projection to remain critical,
we instead exhibit compatible low-only connecting paths in that
critical component.

\section{The structure of the four-color density optimum}\label{M-sec:critical}

The optimum $83/114$ has a large family of realizations. We describe the
part of the history graph carrying optimal cycles, then quantify how
closely a near-optimal coloring must follow it. Throughout this section
$0$ denotes a vacancy and labels $1,2,3,4$ are packing colors of $\D{6}$.
For an edge $e=(u,v)$ of the low-four history graph, let $\ell(e)$ be its
output symbol and let $w(e)=1$ if $\ell(e)\in\{1,2,3,4\}$, and $w(e)=0$
otherwise. Its slack under certificate E is
\[
 s(e)=h(v)-h(u)-114w(e)+83\ge0.
\]
An edge is critical if it lies on a directed cycle all of whose edges have
zero slack. Thus the critical graph is the union of the optimal cycles;
zero-slack transient edges are not included merely because their individual
slack is zero. A strongly connected component is a maximal vertex set in
which every vertex can reach every other by a directed path.

\begin{theorem}\label{M-thm:critical-structure}
The critical graph has two strongly connected components, with respectively
$308$ vertices and $318$ edges, and $304$ vertices and $314$ edges.
Every infinite critical path has label densities
\[
 (\rho_0,\rho_1,\rho_2,\rho_3,\rho_4)
   =\frac1{114}(31,48,16,12,7).
\]
It admits nonperiodic paths. \end{theorem}

\begin{proof}
The calculation uses the complete $800532$-vertex, $1755623$-edge history
graph from certificate E. Its zero-slack subgraph has $1031501$ edges.
The verification data package lists the $612$ critical vertices and $632$ labelled edges,
checks both stated components to be strongly connected, and supplies an
integer function $g$ on the whole graph satisfying
\begin{equation}\label{M-eq:critical-bad-edge}
 b(e)\le99s(e)+g(v)-g(u),
\end{equation}
where $b(e)=0$ on the listed edges and $b(e)=1$ elsewhere.
Every one of the $1755623$ inequalities is checked in integer arithmetic.
Summing on any zero-slack cycle forces $\sum b(e)=0$, proving that no
critical cycle was omitted. Conversely every listed edge is on such a
cycle, by strong connectivity and its checked zero slack.

The verification data package also supplies five integer potentials $P_i$ on the critical
graph such that, on every critical edge,
\[
 P_i(v)-P_i(u)=114\mathbf1_{\ell(e)=i}-k_i,
 \qquad(k_0,k_1,k_2,k_3,k_4)=(31,48,16,12,7).
\]
The finite potential ranges prove all five density assertions by
telescoping, without requiring periodicity.

Each component contains two distinct length-$114$ closed paths based at
the same vertex, with different first labels. Their explicit words and
edges are supplied and checked. Arbitrary two-sided concatenations of the
two words remain legal and optimal. An aperiodic sequence of block choices
gives a nonperiodic label sequence: a periodic label sequence would induce
a periodic sequence of aligned $114$-blocks.

\end{proof}

Supplement Section~\ref*{S-supp:critical-counts} gives the exact entropy
and period counts of the critical subsystem. An optimal-density sequence
may also have zero-density defects. We next control those defects.

\begin{theorem}\label{M-thm:critical-stability}
For every length-$N$ legal low-color history path, let $M_N$ be its number
of nonzero labels and $B_N$ its number of noncritical edges. Then
\begin{equation}\label{M-eq:critical-stability}
 B_N\le99(83N-114M_N)+31930.
\end{equation}
If $M_N/N\to83/114$, then $B_N/N\to0$, and the individual label
densities are $(31,48,16,12,7)/114$ even in the presence of these defects.
\end{theorem}
\begin{proof}
In~\eqref{M-eq:critical-bad-edge}, set $H=99h+g$. Its range is $0\le H\le31930$. Substitution gives
\[
 b(e)\le99(83-114w(e))+H(v)-H(u),
\]
and summation proves~\eqref{M-eq:critical-stability} uniformly over intervals.
Extend each $P_i$ by zero outside the critical vertices. Its global range
is $R_i$, where
\[
 (R_0,R_1,R_2,R_3,R_4)=(251,144,234,108,133).
\]
The edge error in its identity is zero on critical edges and has absolute
value at most $C_i=\max(k_i,114-k_i)+R_i$ otherwise. Consequently
\[
 |114n_i(N)-k_iN|\le R_i+C_i B_N.
\]
This proves the last assertion. If only an upper density equals $83/114$,
the conclusion about $B_N/N$ is initially along the subsequence attaining
that upper density.
\end{proof}

\section{High colors and a quantitative density tradeoff}\label{M-sec:critical-capacity}

For $i\ge5$, legal consecutive integer gaps are at least $6i-9$, and
$2(6i-9)>6i$. Since every forbidden displacement has absolute value at
most $6i$, checking the preceding occurrence suffices for color $i$.
Take the product of the critical low-color graph with an age in
$\{1,\ldots,6i+1\}$. Every step increases the age, capped at $6i+1$;
a vacancy edge may instead carry color $i$ when its current age has
graph distance greater than $i$, after which the next age is $1$.

\begin{theorem}\label{M-thm:critical-capacities}
The largest possible density of one additional color $i$ on a critical
low-color arrangement is the value $\gamma_i$ below. Every value is attained
by a periodic partial coloring. The interval count satisfies
$n_i\le\gamma_i N+r_i$ with the listed boundary constant.
\begin{center}
\begin{tabular}{rccrcc}
\toprule
$i$&$\gamma_i$&$r_i$&$i$&$\gamma_i$&$r_i$\\\midrule
5&$7/171$&$265/171$&13&$4/285$&$322/285$\\
6&$5/152$&$24/19$&14&$1/76$&$20/19$\\
7&$5/171$&$22/19$&15&$2/171$&$236/171$\\
8&$4/171$&$25/19$&16&$5/456$&$47/38$\\
9&$7/342$&$220/171$&17&$1/95$&$20/19$\\
10&$1/57$&$106/57$&18&$3/304$&$43/38$\\
11&$1/57$&$56/57$&19&$3/323$&$368/323$\\
12&$5/342$&$233/171$&&&\\\bottomrule
\end{tabular}
\end{center}
No critical low-color arrangement can be completed using colors at most $18$.
\end{theorem}
\begin{proof}
For each product graph the verification data package supplies integers $p_i,q_i$ and
a potential $H_i$ with
\[
 q_i w_i(e)\le p_i+H_i(v)-H_i(u)
\]
on every edge, where $p_i/q_i=\gamma_i$ and
$(\max H_i-\min H_i)/q_i=r_i$.
All $709995$ edges of the fifteen products are regenerated and checked by the verifier. The product includes every starting age, so its bound applies
to arbitrary finite runs, including runs within an aperiodic arrangement.
The supplied closed walks give periodic words with density $\gamma_i$;
their full integer-graph distance constraints and periodic state/age histories
are verified. This proves exactness, because the construction supplies a compatible past for the attaining cycle.

Writing $\Gamma_K=\sum_{i=5}^K\gamma_i$ and
$R_K=\sum_{i=5}^Kr_i$, one has
\[
 \Gamma_{17}=\frac{293}{1140},\qquad
 \Gamma_{18}=\frac{1217}{4560},\qquad
 \frac{83}{114}+\Gamma_{18}=\frac{4537}{4560}<1.
\]
The uniform boundary bounds therefore exclude a complete $18$-coloring
on a critical low-color arrangement. For example, the low-color boundary
$161/57$ and $R_{18}=5062/285$ give a contradiction already on a critical
run of length $4082$. This does not require periodicity or individual
high-color density limits. For $K=19$ the single-color capacity sum exceeds
the vacancy density; that alone does not construct a simultaneous completion.
\end{proof}

\begin{theorem}\label{M-thm:low-high-tradeoff}
In any complete packing $K$-coloring of $\D{6}$ with $K=17$ or $18$,
the number $M_N$ of positions of colors $1,2,3,4$ in any length-$N$
interval satisfies
\[
 M_N\le\left(\frac{83}{114}-\eta_K\right)N+C_K,
\]
where
\[
\begin{array}{c|cc}
K&\eta_K&C_K\\\hline
17&17/223518090&105398768/37253015\\
18&23/951800250&448816257/158633375
\end{array}
\]
Thus even the upper Banach density of the low colors is strictly below
$83/114$ in either palette.
\end{theorem}
\begin{proof}
Let $B=B_N$ count noncritical edges after all high colors are replaced
by vacancies. The remaining critical edges form at most $B+1$ consecutive
runs of total length $N-B$. Apply the product bounds separately to each
run, with its actual incoming high-color ages. All high colors together
occupy at most $B$ positions on the removed edges. Therefore
\[
 N-M_N\le\Gamma_K N+(1-\Gamma_K+R_K)B+R_K.
\]
Put $A_K=1-\Gamma_K+R_K$ and use~\eqref{M-eq:critical-stability}.
Rearrangement gives
\[
 M_N\le
 \left(\frac{83}{114}-
 \frac{1-83/114-\Gamma_K}{11286A_K-1}\right)N
 +\frac{31930A_K+R_K}{11286A_K-1}.
\]
Substitution of
$A_{17}=3961/228$, $A_{18}=16867/912$ gives the stated fractions.
The constants are uniform in the location of the interval, proving the
upper Banach-density assertion. These are restrictions on possible full
colorings; they do not by themselves exclude every $17$- or $18$-coloring.
\end{proof}

\section{A formula for every high-color capacity}\label{M-sec:all-high}

The finite table above extends to all high colors. Write $\gamma_i$ for
the maximum density of one color $i\ge5$ compatible with a critical
low-color arrangement. The same value is an upper bound along any common
interval sequence on which the low-color density tends to $83/114$.

\begin{theorem}\label{M-thm:all-high}
For every integer $i\ge5$,
\[
 \gamma_i=\frac{1}{6i+\kappa^{(4)}_{i\bmod19}},\qquad
 \frac1{\gamma_{i+19}}=\frac1{\gamma_i}+114,
\]
where the nineteen corrections are given below. Every value is attained
by an actual periodic partial coloring.
\begin{center}\small
\begin{tabular}{rcrcrcrc}
\toprule
$r$&$\kappa^{(4)}_r$&$r$&$\kappa^{(4)}_r$&$r$&$\kappa^{(4)}_r$&$r$&$\kappa^{(4)}_r$\\\midrule
0&$-19/3$&5&$-39/7$&10&$-3$&15&$-9/2$\\
1&$-6$&6&$-28/5$&11&$-9$&16&$-24/5$\\
2&$-39/8$&7&$-39/5$&12&$-18/5$&17&$-7$\\
3&$-33/5$&8&$-21/4$&13&$-27/4$&18&$-20/3$\\
4&$-54/7$&9&$-36/7$&14&$-8$&&\\\bottomrule
\end{tabular}
\end{center}
\end{theorem}

\begin{proof}
Assign to each vertex of a critical component its path-length phase
$\phi\in\mathbb Z/114\mathbb Z$. The verification data package checks
$\phi(v)=\phi(u)+1$ on every edge, and checks the following stronger
mixing property: a length-$114$ walk joins $u$ to $v$ if and only if
$\phi(u)=\phi(v)$. This is verified by Boolean adjacency
powering on the original labelled components. In particular every vertex
has a closed walk of length $114$.

It follows that any length $L\ge114$ joins any two vertices with phase
difference $L$. Indeed, write $L=114k+r$, $0\le r<114$, $k\ge1$.
Take any $r$-step walk to $z$, then $k-1$ closed $114$-walks at $z$,
and finally a $114$-walk to the prescribed endpoint. The last two
vertices have the same phase.

Let $S$ be the phases reached by vacancy-labelled edges. Each component
has $36$ such phases. By Lemma~\ref{M-lem:metric}, a positive gap between
consecutive occurrences of color $i\ge5$ is legal exactly when it is
$6i+k$ with
\[
 k\in\{-9,-8,-4,-3,-2,-1\}\quad\hbox{or}\quad k\ge1.
\]
As explained above, the second preceding occurrence is then farther than
$6i$, so consecutive-gap checks are sufficient.

For $p,q\in S$ and $r=i\bmod19$, let $c_r(p,q)$ be the smallest allowed
$k$ such that $6r+k\equiv q-p\pmod{114}$. It lies in $[-9,114]$.
This defines a complete weighted graph on $36$ phases. Its minimum cycle
mean is $\kappa^{(4)}_r$. The verification data package supplies rational means, integer
potentials and attaining cycles for all nineteen graphs in each component;
all $49248$ edges are regenerated and checked. Both
components give the same table.

For any sequence of consecutive high occurrences, the actual gaps obey
$d_j\ge6i+c_r(p_j,p_{j+1})$. The minimum-mean potential therefore gives
\[
 \sum_{j=1}^m d_j\ge m(6i+\kappa^{(4)}_r)-R_r
\]
with a constant $R_r$ depending only on the finite phase graph. An interval
containing $k$ high occurrences has $k-1$ such gaps, of total length at
most $N-1$. This proves the density upper bound, with a finite uniform
interval constant, for every $i\ge5$.

For attainment when $i\ge21$, take a minimum-mean phase cycle. At every
phase choose a vacancy edge $z_p\to v_p$. For an edge from phase $p$ to
phase $q$, put $d=6i+c_r(p,q)$. Since $d-1\ge116$, the mixing property
gives a length-$d-1$ low walk from $v_p$ to $z_q$; append the vacancy
edge and give that final position color $i$. Concatenating around the
phase cycle produces a closed actual low-history walk with legal high
gaps. If it has $m$ high occurrences, its period is
$m(6i+\kappa^{(4)}_r)$, proving equality. Increasing $i$ by $19$ inserts one
closed $114$-walk into each gap. The reduced denominator of $\kappa^{(4)}_r$
need not equal $m$; the period is integral by the phase congruences.

For $5\le i\le19$, the earlier attaining words give the same formula.
For $i=20$, an additional verified period-$228$ word has two occurrences,
giving $\gamma_{20}=1/114$. Its two gaps are different; a constant gap
$114$ would be forbidden. The phase verifier directly checks
all sixteen short-color words as well as one constructed word for each
residue $21\le i\le39$ in each component. Those finite checks support
the construction; the preceding length-$114$ mixing argument covers every
larger $i$.

Finally split a general near-optimal low path at its noncritical edges.
The stability theorem makes their number $o(N)$. The uniform critical-run
bound for a fixed $i$, summed over these runs, preserves the same limiting
upper bound. This proves the first paragraph's assertion with sparse defects.
\end{proof}

The period $19$ comes from $114/\gcd(114,6)$. The formula expresses an
eventual phase reduction of an unbounded age parameter, with the finite
initial range closed by actual words. It concerns a single high color at
a time; simultaneous high-color optima require additional constraints.

\section{Complete joint density regions}\label{M-sec:density-regions}

We can also solve joint optimization problems. All simultaneous densities
in this section use the same sequence of intervals whose lengths tend to
infinity; independently chosen
coordinatewise upper-density subsequences are not combined.

\begin{theorem}\label{M-thm:global-density-region}
Let $m$ be the combined density of colors $1,2,3,4$ and $x$ the density
of color $5$. The possible simultaneous limit pairs over all partial
packing colorings of $\D{6}$ form exactly the polygon
\begin{equation}\label{M-eq:global-density-polygon}
 \begin{gathered}
 m\ge0,\quad x\ge0,\quad m\le\frac{83}{114},\quad x\le\frac1{21},\\
 2m+x\le\frac32,\qquad131m+3x\le\frac{191}{2}.
 \end{gathered}
\end{equation}
Every real point is attained with bounded discrepancy on arbitrary
intervals. Every rational point is attained by a periodic partial coloring.
The exact upper frontier is
\[
f(m)=\begin{cases}
1/21,&0\le m\le61/84,\\
3/2-2m,&61/84\le m\le91/125,\\
191/6-(131/3)m,&91/125\le m\le83/114.
\end{cases}
\]
\end{theorem}

\begin{proof}
The low bound follows from certificate E, and $x\le1/21$ from the minimum
legal gap for color $5$. Two potentials on the unrestricted low-history
and color-$5$-age product give
\begin{align}
 2M_N+n_5(N)&\le\frac32N+\frac{13}{2},\label{M-eq:global-low2}\\
 131M_N+3n_5(N)&\le\frac{191}{2}N+\frac{747}{2}.
 \label{M-eq:global-low131}
\end{align}
The product has $24816492$ vertices and $60028037$ edges. Both integer
potentials are checked by a separate transition reconstruction. All attaining
words are also checked directly in the infinite graph. These uniform interval bounds prove the outer
polygon for arbitrary simultaneous subsequential limits, including shifted
intervals. Their intersections give its three upper vertices
\[
 A=\left(\frac{61}{84},\frac1{21}\right),\quad
 B=\left(\frac{91}{125},\frac{11}{250}\right),\quad
 C=\left(\frac{83}{114},\frac7{171}\right).
\]

For the segment $AB$, the verification data package supplies two words $U,V$ of lengths
$84,250$ with respective counts $(M,n_5)=(61,4),(182,11)$.
They return to the same sufficient history: the last $6i$ low-color
positions for $1\le i\le4$ and the last color-$5$ age capped at $31$.
Both periodic extensions and their directly matching past histories are
verified without relying on a numerical graph vertex identifier. Thus
arbitrary concatenations are legal, without buffers. The words $U^rV^s$
have
\[
 N=84r+250s,\quad M=61r+182s,\quad n_5=4r+11s.
\]
Every rational interior $m$ on $AB$ is attained by clearing denominators
in $s/r=(84m-61)/(182-250m)$; the endpoint words handle $r=0$ or $s=0$.

For $BC$, another pair of shared-history words of lengths $342,250$ has
counts $(249,14),(182,11)$. Their concatenations have
\[
 N=342r+250s,\quad M=249r+182s,\quad n_5=14r+11s.
\]
Writing $\delta=83/114-m$, the ratio
$s/r=19494\delta/(1-14250\delta)$ attains every rational interior point
of $BC$. In \texttt{packing-baseline-evidence.zip}, the directory
\path{baseline_evidence/project/packing/research/joint_high_2026-09-05/}
contains the two files \path{MIDDLE_FACE_ENDPOINT_CYCLES.json}
(keys $U,V$) and \path{PERIODIC_FRONTIER_FAMILY.json} (keys $A,B$).
They give the words on $AB$ and $BC$, respectively; the larger graph searches used to locate them
are not needed to verify this construction.

Lemma~\ref{M-lem:density-realization}(1) supplies the real points of
both segments with uniformly bounded discrepancy. For
$0\le m\le61/84$, thin the low-colored occurrences of the word at
$A$, leaving color $5$ unchanged; then thin color $5$ to reach any
point below the frontier. Part (3) gives bounded discrepancy and,
for rational targets, periodic realizations. This includes both axes.
\end{proof}

\begin{theorem}\label{M-thm:global-color-five}
Every legal partial coloring satisfies
\begin{equation}\label{M-eq:global-color-five}
 n_5(N)\le\frac7{171}N+
 \frac{131}{3}\left(\frac{83}{114}N-M_N\right)+\frac{249}{2}.
\end{equation}
The coefficient $131/3$ is the least possible coefficient with this
intercept and any finite uniform boundary constant.
\end{theorem}
\begin{proof}
Equation~\eqref{M-eq:global-color-five} is
\eqref{M-eq:global-low131} divided by $3$ and rearranged.
The period-$250$ endpoint word has
\[
 \frac{83}{114}-\frac{91}{125}=\frac1{14250},\qquad
 \frac{11}{250}-\frac7{171}=\frac{131}{42750}.
\]
Applying any proposed inequality to increasingly many repeats of this
word eliminates its fixed boundary constant and forces the coefficient
to be at least $131/3$. This proves optimality.
\end{proof}
The complete period-$250$ sharpness word, with color counts
$(106,35,25,16,11)$, is printed in Supplement
Section~\ref*{S-supp:short-witnesses}.

\begin{theorem}\label{M-thm:three-high-region}
Along intervals on which the low-color density tends to $83/114$, every
common limit vector $(x_5,x_6,x_7)$ lies in
\begin{equation}\label{M-eq:three-high-region}
 \begin{gathered}
 x_i\ge0,\qquad x_5\le\frac7{171},\quad
 x_6\le\frac5{152},\quad x_7\le\frac5{171},\\
 2x_5+x_6\le\frac{13}{114},\qquad
 x_6+x_7\le\frac7{114}.
 \end{gathered}
\end{equation}
Every point in this polytope is an ordinary limiting density vector of
a partial coloring whose low projection lies entirely in the critical
graph. Its nonnegative support function is the maximum of the values at
\[
 a=(112,88,80)/2736,\qquad b=(111,90,78)/2736.
\]
\end{theorem}
\begin{proof}
The coordinate bounds are the single-color capacities. Two
product graphs for pairs of colors give
\[
 2n_5+n_6\le\frac{13}{114}N+\frac{87}{19},\qquad
 n_6+n_7\le\frac7{114}N+\frac{45}{19}.
\]
The verifier regenerates all $806300$ and $1101272$ edges, respectively. For a high reward with critical-run bound $pN/q+r$ and
maximum one-position reward $W$, splitting at $B_N$ noncritical edges gives
\[
 \operatorname{reward}(I)\le\frac pq N+
 \left(W-\frac pq+r\right)B_N+r.
\]
The stability theorem makes $B_N=o(N)$ on the stipulated intervals,
so the inequalities pass to arbitrary near-optimal arrangements, not just
periodic critical paths.

The verification data package gives simultaneous attaining words of lengths $342$ and
$912$, with high counts $(14,11,10)$ and $(37,30,26)$, yielding $a,b$.
Their low histories belong to the same critical component. In units
$1/2736$, any point with sixth coordinate at most $88$ is dominated by
$(112,88,80)$. Otherwise put $\theta=(2736x_6-88)/2\in[0,1]$.
The two joint inequalities show that it is coordinatewise dominated by
\[
 (112-\theta,88+2\theta,80-2\theta)/2736
 =(1-\theta)a+\theta b.
\]
Thus the polytope is exactly the coordinatewise downward closure of this
segment, which proves its stated support function.

To realize the segment, connect the two low histories within their
common critical component and use low-only connecting blocks longer
than $42$, the largest forbidden high-color displacement. The low
projection stays critical and all high-color seams are legal.
The increasing-block construction in Lemma~\ref{M-lem:density-realization}(2)
then gives ordinary densities on both half-lines. Critical color
identities retain low density $83/114$, and part (3) gives the
coordinatewise downward closure. This construction supplies ordinary
limits, without requiring periodicity or bounded interval discrepancy.
\end{proof}

These exact regions identify limits of reweighting the indicated
coordinates. They do not assert that several individually optimal high
colors can be attained simultaneously in a complete coloring, and they
do not raise the current chromatic lower bound of $\D{6}$ beyond $17$.

\section{Colors 1--5 with one additional color}
\label{M-sec:six-labels}

We now optimize the joint density of colors $1,\ldots,5,i$, rather than
requiring the first five colors to remain optimal on their own. Put
$M_5(I)=\sum_{j=1}^5 n_j(I)$ and $N=|I|$. Theorem~\ref{M-thm:six-label-optima} treats $i=6,7,8$.
These are six positive colors; the vacancy symbol $0$ is not counted as a color.

\begin{theorem}\label{M-thm:six-label-optima}
For $i=6,7,8$, the exact maximum joint density of the color set
$\{1,2,3,4,5,i\}$ in a partial packing coloring of $\D{6}$ is,
respectively, $17/21$, $101/126$, and $79/99$. More precisely,
\begin{align}
 M_5(I)+n_6(I)&\le(17N+88)/21,\label{M-eq:six-total}\\
 M_5(I)+n_7(I)&\le(101N+547)/126,\label{M-eq:seven-total}\\
 M_5(I)+n_8(I)&\le(79N+430)/99.\label{M-eq:eight-total}
\end{align}
The corresponding periodic attaining words have lengths $84,252,198$.
The second set omits color $6$, and the third omits both $6$ and $7$.
\end{theorem}

\begin{proof}
Attach to the complete $800532$-state low-history graph the ages of colors
$5$ and $i$, capped at $31$ and $6i+1$. For every $j\ge5$, the allowed
positive gaps between consecutive occurrences of color $j$ are exactly
\[
 \{6j-9,6j-8,6j-4,6j-3,6j-2,6j-1\}\ \cup\ [6j+1,\infty)\cap\mathbb Z.
\]
Since $2(6j-9)>6j$, checking the last occurrence suffices: two older
occurrences cannot conflict with the new one. At every output both ages
advance; only the emitted high color resets to $1$. A vacant low edge
may receive at most one high color. Thus the Cartesian product covers
every legal partial coloring, including noncritical low arrangements.
It is harmless to include tuples that do not encode a common actual past.

For reward $w=1$ on each selected positive label, the verifier
verifies $qw(e)\le p+H(v)-H(u)$ on every product edge. The data are
\[
\begin{array}{c|r|r|c|r}
 i&\text{states}&\text{edges}&p/q&\operatorname{range}(H)\\\hline
 6&918210204&2394752813&17/21&88\\
 7&1067109156&2754921035&101/126&547\\
 8&1216008108&3115089257&79/99&430
\end{array}
\]
Telescoping proves the three finite-interval bounds. The checker rebuilds
the new age transitions from graph-distance balls, using outgoing edges;
the GPU discovery routine instead computes incoming neighbors implicitly.
The complete low graph is the previously reconstructed dependency.

The attaining words below have counts $(n_0,n_1,n_2,n_3,n_4,n_5,n_i)$ equal to
$(16,36,12,8,5,4,3)$, $(50,108,36,24,15,12,7)$, and
$(40,84,28,20,12,9,5)$, respectively. Their infinite periodic extensions
are checked directly against all forbidden signed displacements.
\end{proof}

\begin{theorem}\label{M-thm:sharp-five-eight}
Every interval in a partial packing coloring of $\D{6}$ satisfies
\begin{align}
 4M_5(I)+3n_8(I)&\le(19N+80)/6,\label{M-eq:five-eight-facet}\\
 n_8(I)&\le\frac N{42}+\frac43\left(\frac{65N}{84}-M_5(I)\right)
                    +\frac{40}{9}.\label{M-eq:five-eight-deficit}
\end{align}
The coefficient $4/3$ is the least possible with any fixed boundary
constant. For every simultaneous ordinary limiting low density
$17/22\le m_5\le65/84$, the exact maximum color-$8$ density is
\[
 f_8(m_5)=\frac1{42}+\frac43\left(\frac{65}{84}-m_5\right).
\]
\end{theorem}

\begin{proof}
On the same color-$8$ product, put reward $4$ on colors $1$--$5$ and $3$
on color $8$. A second integer potential has $p/q=19/6$ and range $80$.
All $3115089257$ edge inequalities are checked. Telescoping
and rearranging give the two displayed bounds.

The maximum density of colors $1$--$5$ is $65/84$, also obtained by
maximizing $m+x$ in Theorem~\ref{M-thm:global-density-region}.
At this low density, the new bound gives $x_8\le1/42$; the word $W_*$
below has period $84$ and counts $(M_5,n_8)=(65,2)$, so equality holds.
The $198$-word has
\[
 (m_5,x_8)=\left(\frac{17}{22},\frac5{198}\right),\qquad
 \frac{65}{84}-m_5=\frac1{924},\quad
 x_8-\frac1{42}=\frac1{693}.
\]
Repeating it rules out every coefficient smaller than
$(1/693)/(1/924)=4/3$, whatever the fixed boundary constant.

The endpoint words have lengths $84$ and $198$ and share common
multiple $2772$. Lemma~\ref{M-lem:density-realization}(2), with $48$
vacancies between growing periodic blocks, realizes every convex
combination with ordinary limits. Positions across a buffer are
separated by at least $49>6\cdot8$, so all constraints hold.
\end{proof}

In particular, keeping the low-five density maximal allows joint density
only $65/84+1/42=67/84$. Its unrestricted maximum is larger:
\[
 \frac{79}{99}-\frac{67}{84}=\frac1{2772}>0.
\]
Thus optimizing the low layer first and then filling its vacancies need
not solve the joint problem. Here the exact compensation rate explains
the improvement: sacrificing $1/924$ of low density gains $1/693$ of
color-$8$ density. These partial-coloring results do not themselves give
a full $17$-coloring or rule one out; Technical Supplement~\ref*{S-app:limits} evaluates the resulting
density relaxation.

Supplement Section~\ref*{S-supp:short-witnesses} prints the four
equality words $W_6,W_7,W_8,W_*$ used here, with each row sequence
belonging to one complete period.

\section{The optimal five-color layer}
\label{M-sec:base-five-layer}

The maximum joint density of colors $1$--$5$ is $65/84$. Changing
which low layer is optimized changes both the phase period and the
interactions between higher colors.

\Needspace{13\baselineskip}
\begin{theorem}\label{M-thm:base-five-all-high}
For every $i\ge6$, the exact maximum color-$i$ density compatible with
limiting low-five density $65/84$ in $\D{6}$ is
\[
 \gamma_i^{(5)}=\frac1{6i+\kappa^{(5)}_{i\bmod14}},
 \qquad \frac1{\gamma_{i+14}^{(5)}}=\frac1{\gamma_i^{(5)}}+84,
\]
where the fourteen offsets, in residue order $0,\ldots,13$, are
\[
 \left(0,-6,-6,-\frac{15}{2},-\frac{36}{5},-9,-8,
       -6,-6,-6,-\frac{15}{2},-\frac{36}{5},-9,-6\right).
\]
Every value is attained by a periodic partial coloring. The upper bound
applies along any common sequence of growing intervals on which the
low-five density tends to $65/84$.
\end{theorem}

\begin{proof}
The complete low-five graph has $24816492$ vertices and $60028037$
edges. Its integer potential, with range $301$, satisfies
\[
 s(e)=65+H(v)-H(u)-84w_5(e)\ge0.
\]
Reconstructing every edge proves that the complete
zero-slack cyclic part consists of ten components with $115$ vertices
and $116$ edges and two directed $84$-cycles. The $1328$ internal
edges have constant component rank; all other $7385531$ zero-slack
edges strictly decrease a supplied rank. Forward and reverse
searches verify strong connectivity. These facts exclude any missing
zero-slack cycle without relying on the discovery SCC algorithm.

Each component has period $84$. Exact Boolean adjacency powers verify
that a walk of length $84$ connects two vertices precisely when their
phases agree. Hence every compatible length $L\ge84$ is possible:
take any $L-84$ steps and then the certified $84$-step connection.
The vacant-edge ending phases form sets of size $21$ in the larger
components and size $19$ in the simple cycles.

For phases $x,y$ and residue $r$, let $k_r(x,y)$ be the least member
of $\{-9,-8,-4,-3,-2,-1\}\cup\mathbb Z_{\ge1}$ for which
$6r+k_r(x,y)\equiv y-x\pmod{84}$. The least allowed gap in that
phase class for color $i\equiv r\pmod{14}$ is $6i+k_r(x,y)$.
For each residue, integer potentials on all twelve finite phase graphs
and one attaining cycle certify that their minimum cycle mean is
$\kappa^{(5)}_r$. All $14(10\cdot21^2+2\cdot19^2)=71848$ edge
inequalities are checked exactly.

Write $\kappa^{(5)}_r=p_r/q_r$, $B_i=6iq_r+p_r$, and let $R_r$ be the
largest phase-potential range for that residue. For $k$ high occurrences
in a critical run of length $N$, summing its $k-1$ intervening gap
inequalities gives
\[
 (k-1)B_i\le q_r(N-1)+R_r.
\]
Thus $k\le\gamma_i^{(5)}N+2$ uniformly. For $i=6,\ldots,13$,
the recorded ranges are at most $21$ and $B_i\ge28$; for $i\ge14$,
they are at most $74$ and $B_i\ge75$. These imply the asserted
uniform boundary, including $k=0$.

For $i\ge16$, every required gap minus one is at least $86$, so the
phase cycles lift to actual critical walks using the mixing property.
Mark the chosen vacant edge at each gap endpoint with color $i$.
Nonconsecutive occurrences cannot conflict, and the closed base walk
has density $65/84$. This attains $\gamma_i^{(5)}$. Directly checked
words handle $6\le i\le15$. Increasing $i$ by $14$ inserts one
extra $84$-step base loop per gap, proving the all-parameter recurrence.

Finally let $K$ bound the number of component-rank levels. On an
arbitrary path with low-five count $M_5$, the number of positive-slack
edges is at most $65N-84M_5+301$. Rank constancy inside critical
components and strict descent outside imply that the number $B$ of
noncritical edges satisfies
\[
 B\le K(65N-84M_5+301)+K-1.
\]
At optimal limiting low-five density this is $o(N)$. Splitting at these
edges leaves at most $B+1$ critical runs; the uniform $+2$ bound on each
run and at most one high occurrence per bad edge preserve the same
limiting upper capacity. This proves the theorem for arbitrary common
intervals, not just exactly critical periodic arrangements.
\end{proof}

The capacities for $6$--$17$ satisfy
\[
 \sum_{i=6}^{17}\gamma_i^{(5)}
 =\frac{23227}{105840}=\frac{19}{84}-\frac{713}{105840}.
\]
A complete $17$-coloring
therefore cannot attain this low-five optimum. Together with the new
construction, completion of the optimal five-color layer needs between
$18$ and $20$ colors. This is a conditional conclusion; lower
low-five densities remain possible in unrestricted colorings.

\begin{theorem}\label{M-thm:base-five-box}
At low-five density $65/84$, the complete simultaneous density region
for colors $6,7,8$ is the box
\[
 [0,1/28]\times[0,1/36]\times[0,1/42].
\]
Every real point has an actual realization with uniformly bounded
interval discrepancy. Every rational point has a periodic realization.
\end{theorem}

\begin{proof}
The coordinate upper bounds follow either from the preceding theorem
or by substituting $M_5/N\to65/84$ in
\eqref{M-eq:six-total}, \eqref{M-eq:seven-total} and
\eqref{M-eq:five-eight-facet}. The $252$-period word in Supplement
Section~\ref*{S-supp:short-witnesses} has counts
\[
 (n_0,n_1,\ldots,n_8)=(35,108,36,24,15,12,9,7,6).
\]
All $4104$ signed forbidden-displacement comparisons pass, so it
attains the simultaneous upper corner. Independently thinning the three
high-color occurrence sequences preserves legality. Balanced occurrence
selection has uniformly bounded rounding error, and the original word
is periodic, proving the assertion for real targets. For rational
targets, clear denominators in a common period and retain the specified
integer numbers of occurrences of each color.
\end{proof}

The complete $252$-period corner word is in Supplement
Section~\ref*{S-supp:short-witnesses}.
Its color-$8$ gaps alternate $45$ and $39$; a constant gap of $42$
would be forbidden. The box differs from the nonrectangular region at
optimal low-four density. The low-five optimum occurs at the unique
polygon vertex $(m_4,x_5)=(61/84,1/21)$, so changing the optimized
layer changes the vacancy structure and its simultaneous capacities.

\section{Component choice and nonperiodic density regions}
\label{M-sec:component-competition}

The rectangular region in Theorem~\ref{M-thm:base-five-box} does not extend
to every pair of high colors. Different colors can favor different
components of the optimal low-five layer. This yields an infinite family
for which ordinary limiting, periodic, and bounded-discrepancy density
realizations have different regions.

For this section, an ordinary density means a limit on $[0,N)$ as
$N\to\infty$. The constructions also extend to two-sided colorings
with these densities on both half-lines. Uniform bounded discrepancy
means $n_h(I)=\rho_h|I|+O(1)$ for all intervals $I$, with a constant
independent of their position and length.

The first example uses colors 9 and 10 while the first five colors have
combined density $65/84$. Their simultaneous ordinary density pairs form
\[
 x,y\ge0,\qquad x\le\frac1{48},\qquad y\le\frac2{105},
 \qquad 6x+5y\le\frac3{14}.
\]
Two different optimal low-color components support the two endpoints of the
upper sloping edge. Long blocks from both components realize every intermediate
point with ordinary limits, but a periodic coloring at exactly the optimal
low-five density cannot switch between them. In particular, the rational
midpoint $(41/2016,31/1680)$ is not periodically attainable on this optimal
layer. Theorem~\ref{M-thm:infinite-pair-regions} gives the corresponding statement for all specified
residue classes; it does not rule out periodic approximations with a small
loss of low-color density.

\begin{theorem}\label{M-thm:infinite-pair-regions}
Let $u,v\ge0$ be integers, $i=9+14u$, $j=10+14v$, and put
\[
 X=30i-18,\quad Y=6i-6,\quad U=12j-15,\quad V=6j-4.
\]
At low-five density $65/84$, define
\[
 R_A=[0,1/Y]\times[0,1/V],\qquad
 R_B=[0,5/X]\times[0,2/U].
\]
The complete ordinary limiting region of $(\rho_i,\rho_j)$ is
$\operatorname{conv}(R_A\cup R_B)$, or equivalently
\begin{equation}\label{M-eq:infinite-pair-region}
 x,y\ge0,\quad x\le\frac1Y,\quad y\le\frac2U,\quad
 7XYx+12UVy\le7X+12U.
\end{equation}
The periodic attainable points are exactly
$\mathbb Q^2\cap(R_A\cup R_B)$. Every real point of $R_A\cup R_B$
has a realization with uniformly bounded discrepancy for all color
counts. Conversely, if ordinary density limits exist and the low-five
deficit on $[0,N)$ is bounded above, the density pair is in
$R_A\cup R_B$.
\end{theorem}

The upper segment joins $A=(1/Y,1/V)$ and $B=(5/X,2/U)$. Every strict
point between them is attainable, but is outside both rectangles and
therefore cannot be periodic or have bounded low-five deficit. 

\subsection{Component-wise bounds}

The complete critical graph and its phase functions are those in the
proof of Theorem~\ref{M-thm:base-five-all-high}. Applying the same
phase-gap inequalities separately to each component gives the following
three types. The entries are upper capacities within that component;
simultaneous attainment is established below.
\begin{center}\small
\begin{tabular}{clcc}
\toprule
Type & Number of components & color $i$ & color $j$\\\midrule
$A$ & 6 & $1/Y$ & $1/V$\\
$B$ & 2 & $5/X$ & $2/U$\\
$C$ & 4 & $5/X$ & $1/V$\\\bottomrule
\end{tabular}
\end{center}
The component identifiers and phase data are recorded in Supplement
Section~\ref*{S-supp:component-types}.

For clarity, this smaller calculation does not require a joint age
product. In a component with vacant-edge ending phases $P$, form the
matrix $W_h(p,p')$ of minimum legal gap offsets from $6h$. Its entries
are selected from
\[
 \{-9,-8,-4,-3,-2,-1\}\cup\mathbb Z_{\ge1},
 \qquad 6h+W_h(p,p')\equiv p'-p\pmod{84}.
\]
The matrix depends only on $h\bmod14$. Integer phase potentials satisfy
\[
 qW_h(p,p')-a+H(p)-H(p')\ge0
\]
on every phase pair. The exact mean offsets $a/q$ for the two residue
classes are, respectively,
\[
 A:\ (-6,-4),\qquad
 B:\ (-18/5,-15/2),\qquad
 C:\ (-18/5,-4).
\]
All $2(10\cdot21^2+2\cdot19^2)=10264$ inequalities are checked
exactly, with tight phase cycles. The matrices are rebuilt
from graph-distance balls at colors $23,24$; the stable gap formula
proves the same bounds at every $i,j$ in the theorem, including $9,10$.
Summing between consecutive occurrences gives on each critical run
\begin{equation}\label{M-eq:component-run-bound}
 n_h\le\frac{N}{6h+a/q}+1+
          \frac{\operatorname{range}(H)}{q(6h+a/q)}.
\end{equation}
The retained phase sets and the previously verified complete critical
graph are explicit inputs to this small certificate.

On a component of type $A$, the displayed capacities give the last
inequality of \eqref{M-eq:infinite-pair-region}. On type $B$, its right
side is $35Y+24V=7X+12U$; type $C$ is dominated. Thus every critical
run satisfies the same weighted bound with a fixed boundary constant.
For an arbitrary path with asymptotically optimal low-five density,
the slack and constant-on-component rank argument in
Theorem~\ref{M-thm:base-five-all-high} leaves only $o(N)$ noncritical
edges and $o(N)$ critical runs. Summing
\eqref{M-eq:component-run-bound} preserves the weighted bound after
division by $N$. The coordinate bounds follow from that theorem as well.
This also proves the weighted limsup bound on any common sequence of
intervals whose low-five density tends to $65/84$.

The individual maximal coordinates differ from the alternative
component caps by
\[
 \delta_i=\frac1Y-\frac5X=\frac{12}{XY},\qquad
 \delta_j=\frac2U-\frac1V=\frac7{UV}.
\]
An equivalent form of the new facet is
\begin{equation}\label{M-eq:normalized-component-loss}
 \frac{1/Y-x}{\delta_i}+\frac{2/U-y}{\delta_j}\ge1.
\end{equation}
It quantifies the loss caused by requiring both colors to use a common
distribution of low-layer components. It does not assume that all
competition within a component is explained by these separate bounds.

For the base case, a separately preserved joint product provides another
proof route: $4421890$ states and $4699852$ edges admit an integer
potential of range $184$ proving
$14(6n_9+5n_{10})\le3N+184$ on every critical run. Its outgoing edges
were reconstructed. This larger certificate is unnecessary
for the component-wise proof above.

\subsection{Two explicit corner constructions}

Number positions modulo $84$ by $0,\ldots,83$. Concatenate the two rows
for each of the following words:
\begin{center}\small
\begin{tabular}{cl}
$W_A$ & \texttt{121301014102131215010131421012103101015213}\\
      & \texttt{121401010102131215410131021012100141315210}\\[2pt]
$W_B$ & \texttt{121001014152131210010131421012153101010213}\\
      & \texttt{121401010152131210410131021012153141010213}\\
\end{tabular}
\end{center}
Both are legal low-five words with counts $(19,36,12,8,5,4)$ on
labels $0,\ldots,5$. They are actual $84$-cycles of components $597408$
and $696903$, and direct signed-distance checks verify their periodic
extensions. On their vacant phases, use these cyclic schedules:
\begin{center}\small
\begin{tabular}{ccll}
\toprule
Word & Color & Phases modulo $84$ & Gaps at $u=v=0$\\
\midrule
$W_A$ & $i$ & $6,52,20,66,27,73,34$ & $46,52,46,45,46,45,56$\\
$W_A$ & $j$ & $18,74,46$ & $56,56,56$\\
$W_B$ & $i$ & $17,62,34,80,48$ & $45,56,46,52,53$\\
$W_B$ & $j$ & $3,59,27,78,46,18,69,36$ & $56,52,51,52,56,51,51,51$\\
\bottomrule
\end{tabular}
\end{center}
Within each word the two phase sets are disjoint. All phases are vacant,
all successive phase differences match the listed gaps, and every gap
is legal for its base color $9$ or $10$.

Increase every $i$-gap by $84u$ and every $j$-gap by $84v$. The metric
formula gives $d(0,d+84w)=d(0,d)+14w$ for positive $d$, so these gaps
are legal for their new colors. Their minimum is at least $6h-9$ for
color $h$; twice this exceeds $6h$ here, making nonconsecutive
occurrences safe. The gap-cycle lengths are
\[
 \begin{array}{c|cc}
  &N_i&N_j\\ \hline
  W_A&7Y=84(4+7u)&3V=84(2+3v)\\
  W_B&X=84(3+5u)&4U=84(5+8v).
 \end{array}
\]
Repeat both schedules with common period $\operatorname{lcm}(N_i,N_j)$.
Each occurrence stays on its prescribed vacant phase. Disjoint phase
supports prevent collisions between the two high colors. The resulting
density pairs are exactly $A$ and $B$. This proves all-parameter
attainment using finite schedules, rather than extrapolating sample words.
Balanced occurrence thinning gives every real point in either rectangle
with bounded discrepancy; clearing denominators makes every rational
point in either rectangle periodic.

\subsection{Why the intermediate densities require nonperiodicity}

A periodic low-five optimum is a closed walk with total nonnegative
slack zero. It stays in one critical component, so its high densities
belong to that component's rectangle. This proves the necessity in the
periodic assertion of Theorem~\ref{M-thm:infinite-pair-regions}.

For ordinary-limit attainment, repeat the two corner words to a common
period $L$. At stage $k$, concatenate $k$ such periods of one type;
choose the stage types by a balanced binary sequence of frequency
$\alpha$. Separate stages by $6\max(i,j)$ vacancies. The buffers
prevent every cross-stage conflict. Summation by parts gives an $O(k)$
error in the weighted stage counts, while total length grows as $k^2$.
Buffers and partial stages have vanishing relative length. Thus the
ordinary densities are $\alpha A+(1-\alpha)B$, with low-five density
$65/84$. Thinning gives its whole downward closure, proving
\eqref{M-eq:infinite-pair-region}. Independent copies on the two
half-lines, joined by another buffer, give two-sided realizations.

Finally suppose the low-five deficit on $[0,N)$ is bounded above.
Total nonnegative integer slack is $84$ times that deficit plus a
bounded potential difference, so there are only finitely many
positive-slack edges. Beyond the last one, every remaining noncritical
edge strictly decreases a finite rank; only finitely many such edges
can occur. The path therefore eventually stays in a single critical
component. Any ordinary limiting density pair lies in its rectangle.
This proves the final assertion of
Theorem~\ref{M-thm:infinite-pair-regions}.

\subsection{A logarithmic threshold for the necessary deficit}

The preceding nonperiodicity has a quantitative form. It is a statement
about density realization, not about computation time.

\begin{theorem}\label{M-thm:component-deficit-rate}
Fix a density target $z=(x,y)$ in
$\operatorname{conv}(R_A\cup R_B)\setminus(R_A\cup R_B)$, and put
$\Delta(N)=65N/84-M_5([0,N))$.
Every legal coloring with these ordinary high densities and optimal
limiting low-five density satisfies
\[
 \frac{\Delta(N)}{\log N}\longrightarrow+\infty.
\]
Conversely, for every nondecreasing $f:[1,\infty)\to[1,\infty)$
with $f(N)\to\infty$, there is such a coloring with
\[
 \Delta(N)=O(f(N)\log N).
\]
The coloring and implicit constant may depend on $f$ and the fixed target.
\end{theorem}

\begin{proof}
Outside both rectangles, $x>5/X$ and $y>1/V$. There is a fixed
$\delta>0$ such that in every critical component at least one of its
two capacity bounds is below the target by $\delta$. Its count on
any run of length $L$ is at most $(\rho-\delta)L+C$, with $C$ fixed
by \eqref{M-eq:component-run-bound}. Ordinary convergence bounds the
two relevant prefix errors by $\varepsilon N$ after a sufficiently
large time. For a critical run $[a,b)$ of length $L=b-a$ it follows that
\[
 (\delta-\varepsilon)L\le2\varepsilon a+C.
\]
Thus, for every $\eta>0$, all sufficiently late critical runs, including
truncated final runs, have $L\le\eta a$. Individual noncritical edges
obey the same estimate.

Partition a prefix into maximal critical runs and individual
noncritical edges. If $J(N)$ is its number of pieces, each sufficiently
late piece increases the endpoint by a factor at most $1+\eta$.
Consequently
\[
 \log(N/N_0)\le J(N)\log(1+\eta)+O(1),\qquad
 J(N)/\log N\longrightarrow\infty.
\]
Let $D(N)$ be total nonnegative slack and let $K=24815186$ be the
number of supplied rank levels. As before, the number $b(N)$ of
noncritical edges is at most $K(D(N)+1)$. Hence
\[
 J(N)\le2b(N)+1\le2K(D(N)+1)+1.
\]
Since $D(N)=84\Delta(N)+O(1)$, this proves the lower assertion.

For the converse, choose a point
$z_0=\alpha A+(1-\alpha)B$ dominating $z$ coordinatewise;
$0<\alpha<1$. Repeat the two explicit corner words to a common
period $L$, put $b_0=6\max(i,j)$, and set
$g(T)=\min(f(T),\sqrt T)$. Start with a fixed sufficiently long
all-vacancy prefix. At current length $T_k$, take a positive multiple
$L_k$ of $L$ with
\[
 L_k=T_k/g(T_k)+O(L).
\]
Divide these periods between types $A,B$ in proportions
$\alpha,1-\alpha$, rounding the first number to an integer. Append
the $A$ periods, $b_0$ vacancies, the $B$ periods, and another
$b_0$ vacancies. All seams are safe and
$T_{k+1}=T_k+L_k+2b_0$.

Each completed stage has $O(1)$ high-count error from its target
density times stage length. Its low deficit is exactly
$(65/84)2b_0$. Inside a stage the high error is $O(L_k+b_0)$ and
the low error, beyond the accumulated buffer loss, is $O(L+b_0)$.
For large stages completed by time $N$,
\[
 \log(T_{k+1}/T_k)\ge c/g(T_k)\ge c/f(N)
\]
with fixed $c>0$. Summation bounds their number by $O(f(N)\log N)$.
Using $g(T_k)\le\sqrt N$ in the same argument also bounds it by
$O(\sqrt N\log N)=o(N)$, irrespective of the growth of $f$.
Moreover $L_k/T_k\to0$. The accumulated stage errors and the current
partial-stage error are therefore $o(N)$, proving ordinary convergence
to $z_0$ and low-five density $65/84$. The buffer count proves the
required deficit estimate. Balanced high-occurrence thinning gives $z$
and leaves the low deficit unchanged.
\end{proof}

Thus a logarithmic deficit is impossible for these targets, but its
extra multiplicative factor may diverge arbitrarily slowly. The statement
does not require one coloring to satisfy every such rate simultaneously.

\section{Unrestricted tradeoffs and the seven-color optimum}\label{M-sec:sharp-high-tangents}

The critical-layer capacities need not describe the best use of a higher
color away from that layer. Write $M_5(I)=\sum_{j=1}^5n_j(I)$ and
$\Delta_5(I)=65|I|/84-M_5(I)$.

\begin{theorem}\label{M-thm:sharp-high-tangents}
Every interval $I$ of length $N$ in a partial packing coloring of $\D{6}$
satisfies
\begin{align*}
54M_5(I)+146n_6(I)&\le47N+238,\\
343M_5(I)+381n_7(I)&\le276N+1225,\\
441M_5(I)+660n_9(I)&\le355N+1821.
\end{align*}
In inequalities of the form
\[
n_h(I)\le\gamma_h N+A_h\Delta_5(I)+O_h(1),
\qquad (\gamma_6,\gamma_7,\gamma_9)=(1/28,1/36,1/48),
\]
the least nonnegative coefficients are
\[
A_6=\frac{27}{73},\qquad A_7=\frac{343}{381},\qquad
A_9=\frac{147}{220}.
\]
For each $h$, the exact upper density frontier between the two endpoints
in Table~\ref{M-tab:sharp-tangent-words} is their line segment. Every rational
point on that segment is attained periodically; every real point admits
a coloring with bounded interval discrepancy in both coordinates.
\end{theorem}

\begin{table}[ht]
\centering\small
\begin{tabular}{rccc}
\toprule
$h$ & optimal-layer endpoint & second endpoint & periods\\
\midrule
6 & $(65/84,1/28)$ & $(461/598,11/299)$ & $84,598$\\
7 & $(65/84,1/36)$ & $(657/851,25/851)$ & $252,851$\\
9 & $(65/84,1/48)$ & $(1435/1857,40/1857)$ & $336,1857$\\
\bottomrule
\end{tabular}
\caption{Density coordinates are $(\rho_{1\text{--}5},\rho_h)$.
Each row gives the frontier for its own pair of density coordinates.}
\label{M-tab:sharp-tangent-words}
\end{table}

\begin{proof}
Attach the complete age coordinates for colors $5$ and $h$ to the
future-constraint graph (Technical Supplement~\ref*{S-sec:future-constraints}) of colors
$1,\ldots,4$. The same last-occurrence
argument used above applies to every $h\ge6$. Integer potentials certify
the displayed inequalities on all outgoing product edges. The verifier checks
\[
832\,181\,308,\quad958\,196\,008,\quad1\,210\,225\,408
\]
edges, with potential ranges $238,1225,1821$. The potential weights are
exactly the integer coefficients displayed in the theorem. Telescoping
proves the finite-interval statements; all other high colors can be erased.

The second endpoints are literal periodic words checked against the
signed distance balls of the original infinite graph. In each case
$(\rho_h-\gamma_h)/(65/84-\rho_{1\text{--}5})$ equals the stated $A_h$.
Repeating such a word rules out every smaller coefficient regardless of
the fixed boundary constant.

The retained endpoint words also have rotations, with reversal allowed,
sharing the complete future forbidden-mask state (the states of that graph)
for all six labels.
Each block returns to that state. Therefore arbitrary concatenations of
the two rotated blocks are legal. Integer block multiplicities realize
every rational time mixture. For a real time fraction $\theta$, choose a
balanced binary block sequence of block frequency $p$ satisfying
\[
\theta=\frac{pL_0}{pL_0+(1-p)L_1}.
\]
Its block-count discrepancy is uniformly bounded on every interval of
block indices. Total lengths and color counts are fixed linear combinations
of those counts. Partial end blocks have bounded length, proving bounded
discrepancy on every interval of positions. The global supporting
inequality then proves that this attained segment is the exact upper
frontier on its stated low-density range.
\end{proof}

An exact later-residue calculation is useful when assessing parameter
extensions. The same complete outgoing check gives
\[
21M_5(I)+196n_{20}(I)\le18N+236,
\]
on $2\,596\,387\,108$ edges. For $i=6,7,9,20$ the second endpoint
words are stored in the baseline archive. Within the directory
\begin{center}\small
\path{baseline_evidence/portable/sharp_tangent_verification_2026-09-06/}
\end{center}
read \path{payload/ENDPOINTS.json}, field \path{second_word}.
The period-$4004$ word has low-five
count $3096$ and color-$20$ count $36$. It proves
\begin{equation}\label{M-eq:color20-transient}
A_{20}=\frac3{28},\qquad
(6\cdot20-8)A_{20}=12>\frac{756}{73}=(6\cdot6-8)A_6.
\end{equation}
Thus equal residue alone does not permit extrapolation from the smallest
color. The next section distinguishes an eventual theorem from its onset.

The three new small-color tangents still do not exclude a $17$-coloring.
After adding them to the earlier density constraints, exact rational
primal and dual vectors give the unchanged relaxation optimum
\[
\frac{8816541035110993}{8804514352847850}>1.
\]
This is the limitation of that explicit collection of inequalities,
not a $17$-coloring or an obstruction to stronger joint-density arguments.

\subsection{The unrestricted seven-color maximum and its stability}
\label{M-sec:cr-seven}

Write
\[
 M_j(I)=\sum_{c=1}^j n_c(I),\qquad
 \rho_c(I)=\frac{n_c(I)}{|I|},\qquad
 m_0=\frac{65}{84},\quad u_0=\frac{211}{252}.
\]
All limiting coordinates in the following statements come from the same
sequence of intervals. The earlier color-$8$ tradeoff permits a gain
larger than the loss of low-five density. For colors $6$ and $7$ together,
the sharp coefficient is smaller than one, and forces a different
conclusion at the maximum total density.

\begin{theorem}\label{M-thm:cr-seven}
The maximum ordinary density and the maximum upper Banach density of
the union of colors $1,\ldots,7$ in a partial packing coloring of
$\D{6}$ are both $u_0$. For common limiting densities
$x=\rho_1+\cdots+\rho_5$ and $u=x+\rho_6+\rho_7$,
\begin{align}
 \rho_6+\rho_7&\le \frac4{63}+
            \frac{137}{177}\left(\frac{65}{84}-x\right),
             \label{M-eq:cr-seven-stability}\\
 0\le\frac{65}{84}-x&\le
            \frac{177}{40}\left(\frac{211}{252}-u\right).
             \label{M-eq:cr-seven-recovery}
\end{align}
The coefficient $137/177$ in \eqref{M-eq:cr-seven-stability} is least
possible, even allowing a fixed uniform boundary constant in the
corresponding finite-interval inequality. The recovery factor $177/40$
is also sharp.

For every sequence of intervals $I_n$ with $|I_n|\to\infty$ and
$M_7(I_n)/|I_n|\to u_0$, the seven individual density ratios converge to
\begin{equation}\label{M-eq:cr-seven-individual}
 (\rho_1(I_n),\ldots,\rho_7(I_n))\longrightarrow
 \left(\frac37,\frac17,\frac2{21},\frac5{84},\frac1{21},
       \frac1{28},\frac1{36}\right).
\end{equation}
No prior existence of the individual limits is assumed.
\end{theorem}

\begin{proof}
The complete finite certificates in
Lemma~\ref*{S-lem:cr-seven-cert} give, on every interval of length $N$,
\begin{align*}
 252M_7(I)&\le211N+805,\\
 548M_5(I)+708(n_6(I)+n_7(I))&\le469N+1954.
\end{align*}
The first is uniform in the interval position. Taking the supremum
over those positions and then the upper-density limit proves the
upper Banach bound. Erasing color $8$ from the period-$252$ corner
word of Theorem~\ref{M-thm:base-five-box} leaves $195+9+7=211$ of the
first seven colors, proving both maximum assertions.

Put $D_5=65N/84-M_5(I)$ and $D_7=211N/252-M_7(I)$. The second
certificate is equivalently
\[
 D_7\ge\frac{40}{177}D_5-\frac{977}{354}.
\]
Together with the uniform low-five bound, this proves
\eqref{M-eq:cr-seven-stability}--\eqref{M-eq:cr-seven-recovery} after
division by $N$ and passage to common limits. Finite $D_5,D_7$ can
be negative; their boundary terms have been retained in this step.

The period-$1084$ equality word has
$M_5=836$, $n_6=39$, $n_7=32$. Its low deficit and high gain satisfy
\[
 65\cdot1084-84\cdot836=236,\qquad
 63(39+32)-4\cdot1084=137.
\]
Their density ratio is $84\cdot137/(63\cdot236)=137/177$.
Repeated periods defeat any smaller coefficient and any fixed
boundary constant. The same positive-deficit word has
$(m_0-x)/(u_0-u)=177/40$, proving sharpness of the recovery factor.

If $M_7(I_n)/|I_n|\to u_0$, the finite recovery inequality and the
uniform low-five bound force $M_5(I_n)/|I_n|\to m_0$. The previously
proved bounds \eqref{M-eq:six-total} and \eqref{M-eq:seven-total} imply
\[
 \limsup_n\rho_6(I_n)\le\frac{17}{21}-\frac{65}{84}=\frac1{28},
 \qquad
 \limsup_n\rho_7(I_n)\le\frac{101}{126}-\frac{65}{84}=\frac1{36}.
\]
Their sum tends to $u_0-m_0=4/63=1/28+1/36$, so both coordinates
converge to these values. Lemma~\ref*{S-lem:cr-low-five-rigidity}, proved
from the complete low-five critical graph and its color coboundaries,
supplies the first five limits. This proves
\eqref{M-eq:cr-seven-individual} on the original interval sequence.
\end{proof}

The coefficient $137/177$ is smaller than one. Consequently, a loss in the
combined density of colors 1--5 cannot be fully recovered by colors 6 and 7.
Approaching the seven-color maximum therefore forces the low-five density
to approach $65/84$. The separate bounds for colors 6 and 7 then force
both of their limiting frequencies, and the low-five component identities
force the remaining five. This conclusion applies along the given common
sequence of intervals. It neither selects a unique coloring nor implies
ordinary densities outside that interval sequence.

\begin{lemma}[The color-$8$ fiber at the seven-color maximum]
\label{M-lem:cr-optimal-eight}
Among ordinary limiting density vectors with $u=u_0$, the complete
color-$8$ fiber is $0\le\rho_8\le1/42$. Every real value has a
realization with bounded interval discrepancy, and every rational
value has a periodic realization. More generally, for common limiting
densities,
\begin{equation}\label{M-eq:cr-eight-nonsharp}
 \rho_8\le\frac1{42}+\frac{59}{10}(u_0-u).
\end{equation}
The coefficient $59/10$ is an upper bound, without a sharpness claim.
\end{lemma}
\begin{proof}
Theorem~\ref{M-thm:cr-seven} forces $x=m_0$ when $u=u_0$.
Equation~\eqref{M-eq:five-eight-facet} then gives $\rho_8\le1/42$;
it also gives the same limsup bound along any maximizing interval
sequence. The already printed period-$252$ three-high-color corner
of Theorem~\ref{M-thm:base-five-box} has
$(\rho_6,\rho_7,\rho_8)=(1/28,1/36,1/42)$. Deleting a balanced
fraction of only its color-$8$ occurrences realizes every lower
value with uniformly bounded counting error. Clearing denominators
gives periodic realizations for rational fractions. Finally combine
the limiting form of \eqref{M-eq:five-eight-deficit} with
\eqref{M-eq:cr-seven-recovery}; $(4/3)(177/40)=59/10$.
\end{proof}

These maximum and stability results leave the stated chromatic bounds
unchanged. Determining the optimal coefficient in
\eqref{M-eq:cr-eight-nonsharp} is a further problem.

\section{A fixed description of all eventual low--high regions}
\label{M-sec:eventual-high-regions}

Here a density pair means simultaneous ordinary limits on $[0,N)$.
For a set $S\subset\mathbb R_{\ge0}^2$, write
\[
\downarrow S=\{(x,y)\ge0:(x,y)\le(u,v)
\text{ coordinatewise for some }(u,v)\in S\}.
\]
Let $\mathcal R_i$ be the density region of the low-five count and color
$i$ in partial colorings of $\D{6}$. Set $m_0=65/84$ and
$C_i=6i+\kappa^{(5)}_{i\bmod14}$, using the offsets of the optimal five-color layer.

\begin{theorem}\label{M-thm:eventual-high-regions}
\textup{(Finite region description.)} There is a computable threshold $i_0$ with the following property.
For each residue $r$ modulo $14$, there is a fixed finite collection
$\mathcal Q_r$ of integer triples $(k_Q,b_Q,d_Q)$, with $k_Q\ge1$ and
$d_Q\ge0$, such that for every $i\ge i_0$, $i\equiv r\pmod {14}$,
\begin{align}
L_Q(i)&=6ik_Q+b_Q,\nonumber\\
V_Q(i)&=\left(m_0-\frac{d_Q}{84L_Q(i)},
                     \frac{k_Q}{L_Q(i)}\right),\nonumber\\
\mathcal R_i&=\downarrow\operatorname{conv}
       \bigl(\{(m_0,0)\}\cup\{V_Q(i):Q\in\mathcal Q_r\}\bigr).
\label{M-eq:eventual-region-formula}
\end{align}
Every $V_Q(i)$ is attained by an actual periodic coloring. For each
nonnegative linear objective, its maximum over $\mathcal R_i$ has a
periodic maximizing coloring. In particular, for $z\ge0$,
\begin{equation}\label{M-eq:eventual-support}
\max_{(x,y)\in\mathcal R_i}(x+zy)
=m_0+\max\left(0,\max_{Q\in\mathcal Q_r}
              \frac{zk_Q-d_Q/84}{L_Q(i)}\right).
\end{equation}
\Needspace{6\baselineskip}
\smallskip\noindent\textup{(Parameter translation.)} For every $i\ge i_0$ and integer $\ell\ge0$,
\begin{equation}\label{M-eq:eventual-projective}
\mathcal R_{i+14\ell}=\downarrow P_\ell(\mathcal R_i),\qquad
P_\ell(x,y)=\left(\frac{x+65\ell y}{1+84\ell y},
                       \frac{y}{1+84\ell y}\right).
\end{equation}
\end{theorem}

The proof below gives a finite computation of the templates
$\mathcal Q_r$ and a sufficient threshold. It does not locate the least
threshold over all residues. The downward closure in
\eqref{M-eq:eventual-projective} allows either color density to be thinned.

\subsection{Eventual exact-length costs}

Restrict the complete low-five graph to its reachable states, obtaining
$G$ with $m$ vertices. Every state is reproduced from its last $31$ symbols
and is flushed by $31$ vacancies. Thus $G$ is strongly connected and its
empty state has a vacancy loop. Its integer potential gives
\[
\sigma(e)=65-84w_5(e)+H(v)-H(u)\ge0,
\qquad\operatorname{range}(H)\le301.
\]
The complete critical certificate identifies every zero-slack cycle;
every critical vertex has a zero-slack closed walk of length $84$.
These retained finite facts are the inputs to the following argument.
Eventual optimal-walk periodicity is standard min-plus theory
\cite{M-CFN12}; we give the elementary special-case reduction needed here.

For every $u,v$ and every length $L\ge207$, there is an $L$-step path
of slack at most $K_0=13756$. Route through the empty state, a fixed
critical vertex, the empty state again, and then $v$, using at most
$124$ edges. At most $83$ empty loops fix the length modulo $84$;
zero-slack $84$-loops at the critical vertex supply the rest. The costly
part has slack at most $65\cdot207+301$.

Let $Z$ be the critical vertices. A path avoiding $Z$ with slack at most
$K_0$ has at most $K_0$ positive-slack edges and $K_0+1$ acyclic zero runs.
Its length is less than $(K_0+1)m$. For a residue $h$ modulo $84$, let
$d_\infty(u,v,h)$ be the minimum slack of a path $u\to v$ with that
length residue which visits $Z$. It is a finite shortest-path problem
on $(\text{vertex},\text{residue},\text{visited }Z)$, and a minimum can
be chosen with at most $168m-1$ edges. The visit flag is preserved by
removing a repeated expanded-state segment. Padding at a critical visit
therefore proves, for
\[
T=\max\{207,13757m,168m\},
\]
that the minimum slack of an exact $L$-step path is
$d_\infty(u,v,L\bmod84)$ whenever $L\ge T$.

Let $V_0$ be the targets of vacant edges and define
\[
d_0(u,v,h)=\min_{\substack{z\to v\\\text{label }0}}
    \{d_\infty(u,z,h-1)+\sigma(z\to v)\}.
\]
For $L\ge T+1$, this is the exact minimum slack of an $L$-step path
ending with a vacant edge. Keeping that last edge is necessary because
it is the position replaced by the next high occurrence.

\subsection{The event graph and the density formula}

Record the low-color history immediately after each occurrence of the additional
color. A segment between consecutive such occurrences becomes one edge of a
finite event graph. Its data record the segment's length residue and the least
possible loss in low-color count. The following construction makes these
edges precise.

For fixed $r$, let $\varepsilon_r(h)$ be the smallest member of
\[
\{-9,-8,-4,-3,-2,-1\}\cup\mathbb Z_{>0}
\]
satisfying $6r+\varepsilon_r(h)\equiv h\pmod {84}$. Form a fixed
multigraph $\mathcal E_r$ on $V_0$, with an edge $(u,v,h)$ of cost
$d_0(u,v,h)$ and offset $\varepsilon_r(h)$. Choose $i_0$ so that
$6i_0-9\ge T+1$. The reachable graph has at most
$24816492$ vertices, the size of the full low-history and color-5-age
product. Thus $\lceil(13757\cdot24816492+10)/6\rceil=56900080076$
is one sufficient choice of $i_0$. For every relevant $i$, the shortest legal high gap
in residue $h$ is $6i+\varepsilon_r(h)$; adding $84$ to the gap does
not change its minimum slack. The simple cycles of $\mathcal E_r$
give the finite triples in the theorem by summing edge counts, offsets,
and costs.

\begin{proof}[Proof of Theorem~\ref{M-thm:eventual-high-regions}]
Realize every edge of a simple event cycle by a minimum-slack base path
of length $6i+\varepsilon_r(h)$, marking its final vacancy with $i$.
Concatenation closes the base walk. Consecutive high gaps are legal and
the sum of two exceeds $6i$, so all high comparisons are valid. Its
period is $L_Q(i)$, its high count is $k_Q$, and telescoping slack gives
low count $(65L_Q(i)-d_Q)/84$. This proves the periodic attainments.

Conversely split any periodic coloring at its high occurrences. With
$k$ gaps, total length $L$ and slack $D$, replacing each gap by its
shortest residue representative gives an event walk with
\[
D\ge D_0=\sum d_0(e),\qquad
L\ge L_0=\sum(6i+\varepsilon_r(h)).
\]
The excess of the objective $x+zy$ over $m_0$ is $(zk-D/84)/L$.
If nonpositive it is dominated by the base-only point. Otherwise it is
at most $(zk-D_0/84)/L_0$. Decomposing the event walk into simple cycles
makes this a length-weighted average of the ratios in
\eqref{M-eq:eventual-support}. This proves the support bound for cycles.
For fixed $i$, the original finite low/high graph decomposes an arbitrary
path into cycles and a bounded residual path. Repeating each cycle fixes
its state from its own finite past, so it represents an actual periodic
coloring. Thus the same asymptotic
bound applies to every ordinary limit. The pure high-coordinate bound
follows in the same way, or by letting $z$ tend to infinity.

Lemma~\ref{M-lem:density-realization}(2)--(3), applied to the periodic
vertex words with fixed vacancy buffers, realizes every point of the
downward convex hull with ordinary limits.
Nonnegative support functionals determine that closed convex downward
set, proving~\eqref{M-eq:eventual-region-formula}.

Increasing $i$ by $14\ell$ adds $84\ell k_Q$ positions and $65\ell k_Q$
low occurrences to each cycle template. Hence its point is transformed
by $P_\ell$. A projective map with positive denominator maps a convex
hull onto the convex hull of the transformed points. It fixes $(m_0,0)$.
Finally $P_\ell$ is coordinatewise nondecreasing on
$0\le x\le m_0$, $y\ge0$: the derivative of its first coordinate with
respect to $y$ is nonnegative because $65-84x\ge0$. Therefore taking downward
closure before applying $P_\ell$ does not change its final downward
closure. This proves~\eqref{M-eq:eventual-projective}.
\end{proof}

\subsection{Sharp tangent coefficients}

\begin{theorem}\label{M-thm:eventual-tangent}
Let $A_i$ be the least nonnegative coefficient in
\[
n_i(I)\le |I|/C_i+A_i\Delta_5(I)+O_i(1).
\]
There are computable rational constants $\lambda_r\ge0$ such that
\[
A_i=\frac{\lambda_{i\bmod14}}{C_i}\quad(i\ge i_0).
\]
They are zero for $r=5,12$ and positive for every other residue. If
$\lambda_r>0$, a fixed event-cycle template gives a periodic equality
witness for every $i\ge i_0$ in that residue.
\end{theorem}

\begin{proof}
Give an event edge reward $a(e)=\kappa^{(5)}_r-\varepsilon_r(h)$ and cost
$d(e)=d_0(e)$. A zero-cost event cycle has nonpositive reward, since its
periodic realization is on the optimal base layer with high capacity
$1/C_i$. Consequently the finite maximum
\[
R_r=\max\left(0,\max_{Q:d_Q>0}
\frac{k_Q\kappa^{(5)}_r-b_Q}{d_Q}\right)
\]
exists and $a-R_rd$ has no positive cycle. Choose a potential $\Phi$
with $a(e)-R_rd(e)\le\Phi(v)-\Phi(u)$. Summing between the first and
last high occurrences of an interval, and using nonnegativity of the
omitted base slack, gives
\[
n_i(I)\le\frac{N}{C_i}+\frac{84R_r}{C_i}\Delta_5(I)
 +\frac{C_i+301R_r+\operatorname{range}(\Phi)}{C_i}.
\]
The same constant covers an interval with no high occurrence. If
$R_r>0$, an attaining event cycle has positive slack and realizes exactly
that slope, proving sharpness. Thus $\lambda_r=84R_r$.

When $\kappa^{(5)}_r=-9$, the elementary minimum-gap bound already gives
$A_i=0$. Otherwise the all-vacancy low layer with high occurrences every
$6i-9$ positions has density greater than $1/C_i$, ruling out a zero
coefficient. The offset table gives exactly residues $5,12$ in the first
case.
\end{proof}

The general threshold here is conservative. Equation~\eqref{M-eq:color20-transient}
shows why one cannot infer the normalized constants from the smallest
representative of each residue. Theorem~\ref{M-thm:cr-residue8} below determines
the complete region and earliest stable tail for residue $8$ modulo $14$.
Computing the fixed event costs efficiently and finding the first stable
representatives in the remaining residues are further questions.

\subsection{A complete region with its first stable index}
\label{M-sec:cr-residue8}

The finite description of Theorem~\ref{M-thm:eventual-high-regions}
has the following explicit instance. All parameters $t$ in this
subsection are nonnegative integers. For a fixed
$i\in\{8,22,36,\ldots\}$, let $\mathcal R_i$ be the ordinary
limiting density region for colors $1,\ldots,5$ and color $i$,
with coordinates $(x,y)=(\rho_{1\text{--}5},\rho_i)$. Define
\begin{equation}\label{M-eq:cr-residue8-polygon}
 \Gamma_i=\left\{(x,y):
 \begin{array}{l}
 x,y\ge0,\quad x\le65/84,\quad y\le1/(6i-9),\\
 168x+(18i-18)y\le133,\\
 84x+(24i-30)y\le69,\\
 420x+(192i-255)y\le357
 \end{array}\right\}.
\end{equation}

\begin{theorem}\label{M-thm:cr-residue8}
For every $i=36+14t$, $t\ge0$, one has
$\mathcal R_i=\Gamma_i$. This polygon has seven vertices, each
attained by a periodic partial coloring. Every real point of the
polygon is attained with ordinary limits.

The least $j\in\{8,22,36,\ldots\}$ for which
$\mathcal R_{j+14t}=\Gamma_{j+14t}$ for all integers $t\ge0$ is
$j=36$. In particular
\[
 \mathcal R_{22}\subsetneq\Gamma_{22},\qquad
 \left(\frac{317}{410},\frac1{123}\right)
          \in\Gamma_{22}\setminus\mathcal R_{22}.
\]
The displayed point is excluded even as a limit of densities of
arbitrary legal finite intervals.
\end{theorem}

\begin{proof}
The actual prefix-clock certificates in
Lemma~\ref*{S-lem:cr-residue8-supports} prove the first two sloping
bounds in \eqref{M-eq:cr-residue8-polygon} for every $i=8+14t$,
and the third for every $i=22+14t$. Together with the low-five
bound and the minimum-gap bound they give the outer inclusion for
every $i\ge36$ in this residue.

Lemma~\ref*{S-lem:cr-residue8-vertices} calculates the seven vertices
and supplies their periodic realizations. Its constructions insert
verified $84$-step low-only return words in the designated high
gaps; their low counts and legal gap offsets give the parameter
formulas exactly. For each fixed $i$, growing blocks of these
periodic words, separated by fixed vacancy buffers, realize every
convex combination with ordinary limits. Thus the whole polygon
is attained in the sense of ordinary density.

At $i=22$ all three upper inequalities still hold, but
Lemma~\ref*{S-lem:cr-exclude22} excludes the displayed cap corner
using two nonnegative slacks and a complete certificate on their
common zero-slack graph. Hence a start at either $8$ or $22$
cannot cover its entire subsequent residue-class tail. The equality
proved from $36$ gives the stated least start.
\end{proof}

The start is measured within the residue class. If an unrestricted
integer cutoff $I$ is used in the condition ``every $i\ge I$ with
$i\equiv8\pmod {14}$'', its least value is $23$. No assertion
about the full region at the isolated index $8$ is needed. Its first
exposed face is already determined by
Lemma~\ref*{S-lem:cr-first-face}, which applies to every $i=8+14t$.

\section{Applications to packing chromatic numbers}\label{M-sec:constr}

\begin{theorem}\label{M-thm:main}
The following inequalities hold:
\[
17 \le \chirho(\D{6}) \le 20; \qquad
18 \le \chirho(\D{8}) \le 22; \qquad
\chirho(\D{9}) \le 17 .
\]
\end{theorem}
\begin{proof}[Proof of the lower bounds]
Suppose an interval of length $N$ is completely colored. For $\D{6}$,
certificate E bounds the contribution of colors $1$--$4$ by
$(83/114)N+O(1)$. Successive occurrences of color $i\ge5$ have
separation at least $6i-9$ by Lemma~\ref{M-lem:metric}. Thus a
16-coloring of all of $\mathbb Z$ would give, on dividing by $N$ and
letting $N$ increase,
\[
 1\le \frac{83}{114}+\sum_{i=5}^{16}\frac{1}{6i-9}
   =\frac{869550981179}{873408586050}<1 .
\]
For $\D{8}$, add certificates A--D with positive multipliers
$1/227,4/755,10713/171385,1/151$. Each of the first five counts has
coefficient one and the sixth has coefficient at least one. Their sum
is therefore at most $(20608891/24679440)N+O(1)$. Colors
$i\ge7$ have separation at least $8i-20$, so a 17-coloring would give
\[
 1\le \frac{20608891}{24679440}+\sum_{i=7}^{17}\frac{1}{8i-20}
   =\frac{79832901112650281}{79834202944095600}<1 .
\]
Supplement Sections~\ref*{S-supp:d6-lower} and~\ref*{S-supp:d8-lower}
give the interval constants and the two spacing calculations.
\end{proof}

\begin{lemma}[infinite extension]\label{M-lem:ext}
Let $t$ divide $P$ and let $w$ be a word of length $P$, extended periodically
to $\mathbb{Z}$. For every occurring color $i$, its minimum distance between
distinct equally colored vertices is obtained by testing all residues $u$ and
all displacements $1 \le \delta \le P$ with
$w_u = w_{(u+\delta) \bmod P} = i$, using Lemma~\ref{M-lem:metric}.
\end{lemma}

\begin{proof}
Every occurrence repeats at displacement $P$ and distance $P/t$. Since each edge changes
the coordinate by at most $t$, $d(0,\delta)\ge\lceil\delta/t\rceil>P/t$ for $\delta>P$, so such
displacements cannot reduce the minimum. Interchanging two
endpoints removes the need to consider negative displacements. The test
includes $\delta = P$, so distinct copies of the same residue are included.
\end{proof}

\begin{proof}[Proof of the upper bounds in Theorem~\ref{M-thm:main}]
Technical Supplement~\ref*{S-supp:periods} gives words with $P=2520$ for $\D{6}$, $P = 720$ for $\D{8}$,
and $P = 360$ for $\D{9}$. Lemma~\ref{M-lem:ext} and the exact
minima in Technical Supplement~\ref*{S-supp:minima} show that every minimum exceeds its color label. They
establish the upper bounds $20$, $22$ and $17$ in Theorem~\ref{M-thm:main}. A separate
implementation also checks all displacements $a + tb$ with $|a| + |b| \le i$,
including nonzero multiples of $P$. Both checks examine the infinite periodic
extension, not merely an induced graph on one period.
\end{proof}

For $\D{6}$, the new construction uses the optimal five-color density
$65/84$. Colors $6,7,8$ simultaneously have densities $1/28,1/36,1/42$.
Periodic residue rules and two local patterns for color $3$ describe its
low part; the remaining entries are printed in Technical Supplement~\ref*{S-supp:periods}. The word was
found by GPU local search and checked in the infinite graph.
The period-$1368$ word uses $21$ colors and has optimal
four-color density $83/114$. It still bounds completion of that distinct
four-color layer between $19$ and $21$, together with
Theorem~\ref{M-thm:critical-capacities}. The new $20$-color word optimizes
the five-color layer instead.

The other constructions were found by periodic local search with the period fixed in advance,
color $1$ fixed at residues $0,2,4,6$ modulo $9$ for $\D{8}$ and at even positions for $\D{9}$.
These restrictions only guided the search; the lower-bound proofs do not use them. An earlier
search had produced a $23$-coloring of $\D{8}$ of period $360$, which is superseded by the word
of Technical Supplement~\ref*{S-supp:periods}.

%% REVISED_2026_09_04

\section{Verification and data availability}\label{M-sec:repro}

The data and programs are archived at
\url{https://doi.org/10.5281/zenodo.22562359}. Table~\ref{M-tab:scalar-certificates}
lists the five scalar occupancy certificates. The seven-color result uses
the complete three-age model and the low-five color identities
(Supplement~\ref*{S-supp:cr-certificates}--\ref*{S-supp:cr-individual-rigidity}).
For the residue-eight region, Supplement~\ref*{S-supp:cr-prefix-clock}--\ref*{S-supp:cr-onset}
gives the prefix-clock inequalities, the words realizing the vertices
and the exclusion at index 22.

\subsection*{Use of generative AI}
Generative AI tools, including OpenAI Codex, assisted with mathematical
analysis, research programming, literature review and manuscript preparation.
The author takes responsibility for the mathematics, computations and text.
\enlargethispage{2\baselineskip}

\clearpage
\hypertarget{supplement-start}{}
\setcounter{section}{0}\setcounter{subsection}{0}
\setcounter{equation}{0}
\setcounter{table}{0}\setcounter{figure}{0}
\renewcommand{\MainPrefix}{Main }\renewcommand{\SuppPrefix}{}
\section*{Technical supplement}
Sections S2 and S3 prove the two chromatic lower bounds. The remaining
sections give the density certificates, the periodic constructions and the
word data used for the chromatic upper bounds.

\renewcommand{\thesection}{S\arabic{section}}
\renewcommand{\theStheorem}{S\arabic{Stheorem}}
\renewcommand{\theequation}{S\arabic{equation}}
\renewcommand{\thetable}{S\arabic{table}}

\section{A worked certificate example}\label{S-supp:worked-example}
In the notation of Section~\ref*{M-sec:certificate-principle} of the main paper, take $t=2$ and $C=\{1,2\}$ with unit rewards. Then $F_1=\{1,2\}$,
$F_2=\{1,2,3,4\}$, the horizons are $2$ and $4$, and the history automaton has $13$ states and
$21$ transitions. Its maximum cycle mean is $1/2$: with $p=1$, $q=2$, the potential $h$ that is
$0$ on the empty state and on the states $\{(2,3)\}$, $\{(2,4)\}$, $\{(1,2)\}$ (a single occurrence
of the indicated color at the indicated age), $2$ on $\{(1,1),(2,2)\}$, $\{(1,1),(2,3)\}$,
$\{(1,2),(2,1)\}$, and $1$ on the remaining six states satisfies $h(v)-h(u)\ge 2w(u,v)-1$ on every
transition, and the periodic word $1\,0\,0\,1\,0\,2$ is legal with mean reward $1/2$. Hence in every
packing coloring of $\D{2}$ the colors $1$ and $2$ together occupy at most $N/2+1$ positions of any
interval of length $N$, although separately color $1$ can occupy a third and color $2$ a fifth of
the positions.

\section{Excluding sixteen colors in \texorpdfstring{$\D{6}$}{D(1,6)}}\label{S-supp:d6-lower}

Certificate E and Lemma~\ref*{M-lem:cert} give
\begin{equation}\label{S-eq:d6}
n_1 + n_2 + n_3 + n_4 \le \tfrac{83}{114} N + \tfrac{161}{57}.
\end{equation}
Its attaining period, a periodic partial coloring using only the colors $1,2,3,4$, has
$48,16,12,7$ occurrences of these colors and $31$ vacancies. Hence $83/114$ is the exact maximum
joint upper occupancy of the four colors $1,2,3,4$ among partial packing colorings of $\D{6}$.
A complete $21$-coloring of period $1368$ attains this four-color density,
as recorded in Section~\ref*{M-sec:constr}. \cite[Lemma~8]{M-EHL12} already uses joint low-color occupancy and
high-color spacing for this graph; \eqref{S-eq:d6} strengthens that particular
estimate.

For $i \ge 5$ every positive displacement $\delta < 6i-9$ has distance at most
$i$. Indeed $6i-9 = 6(i-2)+3$, so in Lemma~\ref*{M-lem:metric} either $q \le i-3$,
when $\min(r,7-r) \le 3$, or $q = i-2$ and $r \le 2$. Thus consecutive
occurrences of color $i$ have integer separation at least $6i-9$, and
\begin{equation}\label{S-eq:gap6}
n_i \le \frac{N}{6i-9} + 1 .
\end{equation}

If sixteen colors covered an interval, \eqref{S-eq:d6} and \eqref{S-eq:gap6} would
imply $N \le \rho_6 N + C_6$ with
\[
\rho_6 = \frac{83}{114} + \sum_{i=5}^{16} \frac{1}{6i-9}
       = \frac{869550981179}{873408586050} < 1, \qquad
C_6 = \frac{161}{57} + 12 = \frac{845}{57}.
\]
At $N = 3357$ the difference $N - \rho_6 N - C_6$ equals
$231148633/97045398450 > 0$, a contradiction. This proves
$\chirho(\D{6}) \ge 17$. This $N$ is a sufficient interval length, not a
claimed optimal one. \qed

The improvement is located precisely in the low-color occupancy input, not in
the surrounding argument. Both \cite{M-SV15} and this section instantiate the
same scheme as \cite[proof of Lemma~8]{M-EHL12}, which applies the density-sum
inequality of \cite[Lemma~7]{M-EHL12}, due to \cite{M-FKL09}, to a computed joint
occupancy bound on an initial segment of colors plus the elementary spacing
bound \eqref{S-eq:gap6} on each remaining color. With the estimate
$d(1,\dots,8) \le 89/101$ used in
\cite{M-SV15}, the same scheme applied to sixteen colors gives
\[
\frac{89}{101} + \sum_{i=9}^{16} \frac{1}{6i-9}
 = \frac{310940202439}{308441493675} > 1 ,
\]
so that input does not exclude a sixteen-coloring; it excludes only fifteen,
which is exactly the bound obtained there. Substituting the certified value
$83/114$ for the four colors $1,2,3,4$ supplies the missing margin. The same
computation also delimits the method:
$83/114 + \sum_{i=5}^{17} 1/(6i-9) > 1$, so certificate E and the spacing bound
alone cannot exclude a seventeen-coloring.

\section{Excluding seventeen colors in \texorpdfstring{$\D{8}$}{D(1,8)}}\label{S-supp:d8-lower}

Certificates A--D and Lemma~\ref*{M-lem:cert} yield, on every common interval of
length $N$,
\begin{align}
208 n_1 + 227 n_2 + 227 n_4 &\le 137 N + 560, \tag{A}\\
4 n_1 + 5 n_5 &\le 2N + 8, \tag{B}\\
n_1 + 16 n_6 &\le N + 15, \tag{C}\\
151 n_3 + 147 n_5 &\le \tfrac{3445 N + 50077}{144}. \tag{D}
\end{align}
For completeness, (B) and (C) have direct proofs in Technical Supplement~\ref{S-supp:elementary}. Multiply
(A)--(D), in order, by $1/227$, $4/755$, $10713/171385$ and $1/151$. The
coefficient of each $n_i$ for $1 \le i \le 5$ is exactly one. The coefficient
of $n_6$ is $171408/171385 > 1$, whose excess can be discarded since
$n_6 \ge 0$. Hence
\begin{equation}\label{S-eq:d8}
\sum_{i=1}^{6} n_i \le \frac{20608891}{24679440} N
   + \frac{141906691}{24679440}.
\end{equation}

For $i \ge 7$ every displacement $\delta < 8i-20 = 8(i-3)+4$ is forbidden to
color $i$. If $\delta = 8q+r$ and $q \le i-4$, the metric is at most
$q + 4 \le i$; the only other possibility is $q = i-3$ with $r \le 3$, again
giving distance at most $i$. Therefore
\begin{equation}\label{S-eq:gap8}
n_i \le \frac{N}{8i-20} + 1 .
\end{equation}

In a seventeen-coloring, \eqref{S-eq:d8} and \eqref{S-eq:gap8} would give
$N \le \rho_8 N + C_8$, where
\[
\rho_8 = \frac{20608891}{24679440} + \sum_{i=7}^{17} \frac{1}{8i-20}
       = \frac{79832901112650281}{79834202944095600} < 1,
\]
\[
1 - \rho_8 = \frac{1301831445319}{79834202944095600}, \qquad
C_8 = \frac{413380531}{24679440}.
\]
At $N = 1027186$, exact rational arithmetic gives $\rho_8 N + C_8 < N$. This
proves $\chirho(\D{8}) \ge 18$. As in Technical Supplement~\ref{S-supp:d6-lower}, this is an unrestricted
infinite-graph bound, not a conclusion from a failed search over periodic
colorings. \qed

\section{Elementary proofs of (B) and (C)}\label{S-supp:elementary}

Every nine consecutive vertices of $\D{8}$ contain a cycle of length nine, so
at most four can have color $1$. Applying this to disjoint nine-vertex blocks
and using the path bound on a remainder gives, for $m = 9q + r$,
$0 \le r \le 8$,
\begin{equation}\label{S-eq:c1}
n_1 \text{ on an interval of length } m \;\le\; 4q + \lceil r/2 \rceil .
\end{equation}

To prove (C), consecutive color-$6$ occurrences have integer gap $d \ge 28$ by
Lemma~\ref*{M-lem:metric}. For two consecutive such positions $x < y$, the first
position of $[x,y)$ is not color $1$. Write $d-1 = 9q + r$; then $q \ge 3$.
Equation~\eqref{S-eq:c1} shows that $[x,y)$ contains at least
$d - 4q - \lceil r/2 \rceil = 5q + \lfloor r/2 \rfloor + 1 \ge 16$
positions not of color $1$. If the full interval has $k \ge 1$ color-$6$
positions, its $k-1$ interior gaps give $16(k-1)$ such positions, and its last
color-$6$ position gives one more. Thus $n_1 + 16k \le N + 15$. The case
$k = 0$ is immediate.

For (B), Lemma~\ref*{M-lem:metric} shows that the only allowed gaps between
color-$5$ positions below $28$ are $20, 21, 27$. Two successive gaps cannot
both be $20$: their outer endpoints would differ by $40$, which has graph
distance $5$. For a gap $[x,y)$ write $d = y-x$ and $d-1 = 9q+r$. Its first
site is not color $1$, so \eqref{S-eq:c1} gives
\[
4 n_1([x,y)) + 5 \le 2d - s(d), \qquad
s(d) = \begin{cases} 2q-3 & r \text{ even},\\ 2q-5 & r \text{ odd}.\end{cases}
\]
The slack is $-1$ when $d = 20$, is $1$ when $d = 21$ or $27$, and is at least
$1$ when $d \ge 28$. Nonadjacency of the gaps $20$ implies that the total slack
on any finite list of consecutive gaps is at least $-1$. For $k \ge 1$
color-$5$ positions $z_1 < \cdots < z_k$, summing over the $k-1$ interior gaps
yields $4 n_1([z_1,z_k)) + 5(k-1) \le 2(z_k - z_1) + 1$.
Let $L_0$ be the prefix length before $z_1$ and $L_1$ the suffix length
starting at $z_k$. The latter first position is color $5$. The path bound gives
\[
4 n_1(\text{outside}) \le 4\lceil L_0/2 \rceil
   + 4\lceil (L_1-1)/2 \rceil \le 2(L_0 + L_1) + 2 .
\]
Adding these inequalities and the contribution $5$ of the final color-$5$
position gives $4n_1 + 5k \le 2N + 8$. If $k = 0$, the ordinary path bound is
already stronger. This proves (B) and (C), including the boundary cases
$k = 0$ and $k = 1$. \qed

\section{Finite-state models and integer certificates}
\label{S-sec:future-constraints}

The scalar bounds admit a smaller exact model. This also supplies a
verification path independent of the earlier history encoder.

The future-constraint recognizer is defined in Main
Section~\ref*{M-sec:future-model}; the construction below reconstructs
it directly from the forbidden displacements.

For colors $1$--$4$ of $\D{6}$, this construction has $253874$
productive states and $620172$ edges. A separate direct breadth-first
reconstruction from the empty future constraints, using signed
graph-distance balls, reproduces every state and labelled edge without
reading the old $800532$-state graph or its quotient map.

The potential transfer is proved in Main
Section~\ref*{M-sec:future-model}. The following certificates use that
future-constraint representation.

The three six-label products on this model have respectively
\[
\begin{array}{c|r|r}
i&\text{states}&\text{edges}\\\hline
6&291193478&832181308\\
7&338414042&958196008\\
8&385634606&1084210708
\end{array}
\]
The CPU verifier checks all four quotient potentials on every edge,
including the weighted color-$8$ certificate. Their density values, boundary constants and attaining words
are unchanged. The low-five bound $65/84$ has also been replayed on
the $7870094$-state, $21002450$-edge one-age product. These checks
read neither the old history graph nor the discovery program, and do
not require a GPU. The archive also contains the original certificates on the larger history graphs.

\subsection{Entropy and periods on the optimal low-four graph}\label{S-supp:critical-counts}
Here the critical graph and its two components are those of Main
Theorem~\ref*{M-thm:critical-structure}. The counts concern that critical
subsystem; optimal-density sequences with zero-density defects need not
lie wholly in it.

The \emph{labelled topological entropy} is the exponential growth rate
per position of the number of distinct admissible label blocks. This
counts label sequences, rather than choices of an initial graph vertex.

\begin{Stheorem}
The critical graph for colors $1$--$4$ has entropy $\log(20)/114$.
Every optimal periodic partial coloring lies in one of its two components.
The number of translation classes of least period $114k$ is
\begin{equation}\label{S-eq:optimal-period-count}
 a_k=\frac2k\sum_{d\mid k}\mu(d)20^{k/d},\qquad k\ge1,
\end{equation}
where $\mu$ is the M\"obius function. No other least periods occur.
\end{Stheorem}
\begin{proof}
Each component identified in Main Theorem~\ref*{M-thm:critical-structure}
has $24$ simple directed cycles:
$20$ of length $114$ and four of length $228$. These lists were enumerated
on a suppressed graph and checked on the original graph.
There are exactly four pairs of vertex-disjoint $114$-cycles and no three
disjoint cycles. The latter also follows from the component orders being
less than $342$. The cycle-cover expansion of the determinant therefore
gives, for either adjacency matrix $A_j$,
\[
 \det(I-zA_j)=1-20z^{114}-4z^{228}+4z^{228}=1-20z^{114}.
\]
Hence the spectral radius is $20^{1/114}$. A fixed initial history and
label determine at most one next history, so any label word has at most
$612$ path realizations. Path growth and labelled-word growth thus have
the same exponential rate, proving the entropy formula.

A periodic coloring of density $83/114$ gives a closed history walk whose
total slack is zero. It must lie in one critical component. Conversely a
closed critical walk gives an optimal periodic word. The history is uniquely
determined by the preceding $24$ labels, so closed walks count indexed
periodic words without extra multiplicity from hidden states. The determinant
formula gives
\[
 \operatorname{tr}(A_j^N)=
 \begin{cases}114\cdot20^k,&N=114k,\\0,&114\nmid N.
 \end{cases}
\]
M\"obius inversion, followed by division by the least period, proves
\eqref{S-eq:optimal-period-count}. For $k=1,2,3,4$ the counts are
$40,380,5320,79800$. The counts identify translations only, without
quotienting by reflection.
\end{proof}

\subsection{The three component types}\label{S-supp:component-types}
For the high colors $i=9+14u$, $j=10+14v$, the types and quantities
$X,Y,U,V$ are those of Main Theorem~\ref*{M-thm:infinite-pair-regions}.
This table identifies each of the twelve complete low-five critical
components in the stored history graph.
\begin{center}\small
\begin{tabular}{clcc}
\toprule
Type & Component identifiers & color $i$ & color $j$\\
\midrule
$A$ & $597408,668119,775798,834520,948268,1462729$
    & $1/Y$ & $1/V$\\
$B$ & $696903,943395$ & $5/X$ & $2/U$\\
$C$ & $726571,1002926,1207134,1412245$ & $5/X$ & $1/V$\\
\bottomrule
\end{tabular}
\end{center}

\subsection{Certificates for the seven-color maximum}
\label{S-supp:cr-certificates}

We give the finite certificate hypotheses and the parameter-uniform
deductions used in Theorems~\ref*{M-thm:cr-seven} and
\ref*{M-thm:cr-residue8}. The starting low-four model is the complete
productive future-constraint graph described above, with $253874$
states and $620172$ labelled edges. Its states and all outgoing edges
are reconstructed from the original signed graph-distance
balls. Each state has exactly one vacancy edge.

To retain colors $5,6,7$, append their actual source ages capped at
$31,37,43$. Age one is the position immediately after an occurrence;
an output of that color sets its next age to one, and every other
output increments and saturates the age. An absent preceding occurrence
is represented by the saturated age. The allowed reset ages are
\[
 \begin{array}{c|l}
 \text{color}&\text{allowed capped source ages}\\\hline
 5&21,22,26,27,28,29,31\\
 6&27,28,32,33,34,35,37\\
 7&33,34,38,39,40,41,43.
 \end{array}
\]
The smallest allowed gaps are $21,27,33$ and the largest forbidden
displacements are $30,36,42$. Two allowed consecutive gaps exceed
every forbidden displacement, so all earlier occurrences are harmless
once the last occurrence passes the test. Low-four outputs and a
vacancy advance all ages. Output $5$, $6$ or $7$ uses the low-four
vacancy transition and resets only its own age. These are distinct
outputs; a color-$5$ or color-$7$ reset cannot also be used for color
$6$. The product consequently covers every legal partial coloring,
including intervals obtained by erasing all other colors. Checking
additional tuples without a common reachable past only strengthens
an upper certificate.

\begin{Slemma}\label{S-lem:cr-seven-cert}
For every finite interval $I$, $N=|I|$,
\begin{align}
 252M_7(I)&\le211N+805,\label{S-eq:cr-seven-finite}\\
 548M_5(I)+708(n_6(I)+n_7(I))&\le469N+1954.
       \label{S-eq:cr-seven-weighted-finite}
\end{align}
There is a legal period-$1084$ word with $(M_5,n_6,n_7)=(836,39,32)$
attaining the limiting weighted inequality.
\end{Slemma}
\begin{proof}
The formal three-age product has $253874\cdot31\cdot37\cdot43=12521319554$
tuples, counting every combination of the low mask and the three ages. Put $\ell=1$ on outputs $1,\ldots,5$ and zero otherwise,
and let $h_6,h_7$ indicate outputs $6,7$. The two integer potentials
$F_U,F_W$ obey, on every original product edge $p\to q$,
\begin{align*}
 F_U(q)-F_U(p)&\ge252(\ell+h_6+h_7)-211,\\
 F_W(q)-F_W(p)&\ge548\ell+708h_6+708h_7-469.
\end{align*}
These are the exact checked inequalities. The complete outgoing
inventory for each potential is
\begin{center}\small
\begin{tabular}{@{}lr@{}}\toprule
Actual output family&edges\\\midrule
Low-four output or vacancy&30587503212\\
Color-$5$ reset&2827394738\\
Color-$6$ reset&2368898294\\
Color-$7$ reset&2038354346\\\midrule
Total&37822150590\\\bottomrule
\end{tabular}
\end{center}
For example the first count is $620172\cdot31\cdot37\cdot43$;
the reset counts use $253874$ vacancy edges, seven eligible ages
for the reset color, and every age of the other two colors. Every
saturation boundary is included. The measured potential ranges are
$805$ and $1954$, and the minimum edge slack in every family is zero.
The verifier reconstructs outgoing legality and widens
integer blocks before forming differences. The two potentials are
stored in $37$ age-$6$ layers, allowing source, target and reset
layers to be checked without a full wide-integer tensor. Discovery
code and convergence traces are not used to justify these inequalities.

Summing along the interval's own state path bounds the endpoint
differences by the respective ranges, proving
\eqref{S-eq:cr-seven-finite}--\eqref{S-eq:cr-seven-weighted-finite}
uniformly in interval position. A finite deficit need not be
nonnegative for this argument.

The data package retains the literal $1084$-word. The
original-metric check tests all $15620$ forbidden signed-displacement
comparisons modulo its period. Its counts are as stated, and
\[
 548\cdot836+708(39+32)=469\cdot1084.
\]
This proves the asserted periodic attainment. The period-$252$
corner already printed with Theorem~\ref*{M-thm:base-five-box} supplies
the unweighted maximum and the color-$8$ fiber; its word is reused
without a second transcription.
\end{proof}

\subsection{Individual colors on the optimal low-five layer}
\label{S-supp:cr-individual-rigidity}

\begin{Slemma}\label{S-lem:cr-low-five-rigidity}
If $M_5(I_n)/|I_n|\to65/84$ along a common interval sequence with
$|I_n|\to\infty$, then
\[
 (\rho_1(I_n),\ldots,\rho_5(I_n))
       \longrightarrow(36,12,8,5,4)/84.
\]
\end{Slemma}
\begin{proof}
Use the complete low-five graph and its potential from the proof of
Theorem~\ref*{M-thm:base-five-all-high}. It has $24816492$ vertices and
$60028037$ edges, and
\[
 \sigma(e)=65-84\ell(e)+H(q)-H(p)\ge0,
 \qquad\operatorname{range}(H)=301.
\]
Its entire zero-slack cyclic part consists of twelve components,
with $1318$ vertices and $1328$ internal edges. The computed
rank is constant on their internal edges and strictly
decreases on every other zero-slack edge. Thus no omitted zero-slack
cycle exists. This completeness is a retained full-graph premise;
it does not follow from checking just the small component list.

For each component and each $c=1,\ldots,5$, the additional certificate
gives an integer function $H_c$ with
\begin{equation}\label{S-eq:cr-color-coboundary}
 H_c(q)-H_c(p)=84\,\mathbf1_{\{\text{output}=c\}}-a_c,
 \qquad(a_1,a_2,a_3,a_4,a_5)=(36,12,8,5,4).
\end{equation}
All $5\cdot1328=6640$ integer equalities are checked on the complete
internal edge lists. They determine the color proportions on every
cycle in every retained component, rather than on selected examples.

Take any convergent subsequence of the empirical edge-frequency
vectors of the intervals. The finite graph makes such a subsequence
available. Endpoint imbalance divided by $|I_n|$ vanishes, so its
limit is a nonnegative circulation of total mass one. Telescoping
$\sigma$ shows that its mean slack is zero. Nonnegativity confines
the circulation to zero-slack edges, and cycle decomposition confines
it to the twelve complete critical components. Cross-component
acyclic edges carry no circulation.

On a component of total edge mass $\mu$, summing
\eqref{S-eq:cr-color-coboundary} gives color-$c$ mass $a_c\mu/84$;
the potential difference cancels by flow conservation. The component
masses sum to one, giving the displayed five values. Every convergent
frequency subsequence gives those same color coordinates, hence the
original coordinate sequence converges. No concentration in one
component, unique word or periodicity is required, and this argument
asserts no quantitative per-color convergence rate.
\end{proof}

\subsection{A prefix-clock certificate with a finite actual-age boundary}
\label{S-supp:cr-prefix-clock}

The finite memory records small ages since the last high-color occurrence
exactly, and larger ages through a residue modulo $84$. We call this the prefix clock. The proof must still recover
a bounded potential on actual ages; the cyclic representative alone is not
assumed to be a legal occurrence gap.

Let $\mathcal G_5$ be the complete low-four/age-$5$ product. It has
$7870094$ states and $21002450$ edges; each state has one genuine
vacancy edge. It covers all low-five words. For a residue $r$ modulo
$14$, define
\[
 \mathcal E=\{-9,-8,-4,-3,-2,-1\}\cup\mathbb Z_{\ge1},\qquad
 \varepsilon_r(j)=\min\{e\in\mathcal E:6r+e\equiv j\pmod {84}\}.
\]
For $i\ge6$, the legal positive gaps for color $i$ are exactly
$6i+\mathcal E$. Their minimum is $6i-9$, and twice that minimum
exceeds the largest forbidden displacement $6i$.

\begin{Slemma}\label{S-lem:cr-prefix-clock}
Let $a,b,K$ be positive integers and $c$ an integer. Attach
to $\mathcal G_5$ clock ages $1,\ldots,K+83$, advancing by one
except that $K+83$ returns to $K$. Suppose an integer potential $Q$
satisfies both of the following on their full finite domains:
\begin{itemize}
\item every ordinary low edge advances the clock and has potential
      difference at least $a(84\ell-65)$;
\item at every clock age $j\ge K$, a genuine vacancy edge may be
      marked; it resets the clock to one and has potential difference
      at least $-65a+c-b\varepsilon_r(j)$.
\end{itemize}
Then every finite interval, for every fixed $i\ge6$ with
$i\equiv r\pmod {14}$ and $6i-9\ge K$, satisfies
\begin{equation}\label{S-eq:cr-prefix-bound}
 84aM_5+(6bi+c)n_i\le(65a+b)N+\operatorname{range}(Q)+6bi.
\end{equation}
\end{Slemma}
\begin{proof}
For actual positive ages $d$ let
\[
 j(d)=\begin{cases}d,&d<K,\\K+((d-K)\bmod84),&d\ge K,
 \end{cases}
 \qquad F(p,d)=Q(p,j(d))-bd.
\]
An ordinary edge changes $F$ by at least $a(84\ell-65)-b$.
At a legal high occurrence, $d\ge6i-9\ge K$ and
$d\ge6i+\varepsilon_r(j(d))$. Thus
\[
 F(q,1)-F(p,d)\ge-65a+c-b\varepsilon_r(j(d))+b(d-1)
       \ge(6bi+c)-(65a+b).
\]
The last inequality uses the actual legal gap. A short cyclic
clock representative itself need not be a legal gap for that $i$.

The unfolded $F$ is unbounded below, so it cannot yet give a uniform
endpoint bound. Put $M=6i+1$ and keep actual ages $1,\ldots,M-1$
with one saturated age for all $d\ge M$. Its potential is
\[
 F_{\mathrm{cap}}(p)=\max_{d\ge M}F(p,d).
\]
This maximum is finite and occurs among $d=M,\ldots,M+83$, because
$M\ge K$ and $F(p,d+84)=F(p,d)-84b$. For an ordinary saturated
edge choose an age attaining the source maximum; its next age
belongs to the target fiber, whose maximum is at least its value.
The edge inequality therefore survives. The same argument covers
the exact age $M-1$ entering saturation. For a marked saturated
edge all represented ages are legal, and the target age is one,
so the preceding marked inequality applies to a maximizing source
age as well. A fiber minimum would not justify this step.

Every resulting value is at most $\max Q-b$ and at least
$\min Q-bM$: the saturated maximum is at least $F(p,M)$, and
the exact ages have the same lower bound. Its range is therefore
at most $\operatorname{range}(Q)+b(M-1)$. This finite potential
covers the full usual actual-age recognizer, including an absent
preceding high occurrence at the saturated state. Summing the
ordinary and marked inequalities proves \eqref{S-eq:cr-prefix-bound}.
Other colors can be erased before projection; the constant is
independent of the interval's position.
\end{proof}

\begin{Slemma}\label{S-lem:cr-residue8-supports}
For every finite interval of length $N$ and every $i=8+14t$,
$t\ge0$,
\begin{align}
 168M_5+(18i-18)n_i&\le133N+632+18i,\label{S-eq:cr-prefix-A}\\
 84M_5+(24i-30)n_i&\le69N+319+24i.\label{S-eq:cr-prefix-B}
\end{align}
For every $i=22+14t$, $t\ge0$, one also has
\begin{equation}\label{S-eq:cr-prefix-C}
 420M_5+(192i-255)n_i\le357N+1616+192i.
\end{equation}
\end{Slemma}
\begin{proof}
The exact finite certificates for Lemma~\ref{S-lem:cr-prefix-clock}
have the following parameters:
\begin{center}\small
\begin{tabular}{@{}crrrrcrr@{}}\toprule
Bound&$K$&$a$&$b$&$c$&range&states&edges\\\midrule
A&39&2&3&$-18$&632&960151468&3223386796\\
B&39&1&4&$-30$&319&960151468&3223386796\\
C&123&5&32&$-255$&1616&1621239364&4987592596\\\bottomrule
\end{tabular}
\end{center}
The outgoing verifier tests all ordinary edges at every
clock age and all vacancy marks at all $84$ cyclic ages, including
the $K+83\to K$ boundary. With $V_5=7870094$ and $E_5=21002450$,
the counts are $V_5(K+83)$ states and
$E_5(K+83)+84V_5$ edges. No critical-visit flag or restriction to
tight edges is used. All minimum slacks are zero. The condition
$6i-9\ge K$ holds from $8$ for A,B and from $22$ for C in this
residue. Substitution into \eqref{S-eq:cr-prefix-bound} gives exactly
the three displayed finite inequalities.
\end{proof}

\subsection{Periodic vertices and their all-parameter lifts}
\label{S-supp:cr-vertex-constructions}

Start with a cyclic low-color word containing marked vacant positions. We call
this marked word a scaffold. At specified cuts between successive marks, we
insert verified low-color words that return to the same history state. Filling
the marks with the additional color then produces the parameter-dependent
periodic coloring. The cuts, return-state equalities and resulting gaps are
all part of the construction certificate.

\begin{Slemma}\label{S-lem:cr-residue8-vertices}
For $i\in\{8,22,36,\ldots\}$ put $m_0=65/84$, $g=6i-9$,
$L_F=30i-42=5g+3$ and $L_G=96i-140=16g+4$. The polygon
$\Gamma_i$ of \textup{(\ref*{M-eq:cr-residue8-polygon})} has the seven vertices
\begin{equation}\label{S-eq:cr-seven-vertices}
 \begin{gathered}
 O=(0,0),\quad A_0=(m_0,0),\quad
 E_i=\left(m_0,\frac1{g+3}\right),\\
 F_i=\left(m_0-\frac3{14L_F},\frac5{L_F}\right),\quad
 G_i=\left(m_0-\frac{20}{21L_G},\frac{16}{L_G}\right),\\
 H_i=\left(m_0-\frac{11}{140g},\frac1g\right),\quad
 Z_i=\left(0,\frac1g\right).
 \end{gathered}
\end{equation}
The vertices $E_i,F_i$ have periodic realizations for all $i=8+14t$;
$G_i,H_i$ have periodic realizations for all $i=36+14t$.
Consequently every vertex is periodically attained from $36$, and
every real point of $\Gamma_i$ is attained with ordinary limits
for those indices.
\end{Slemma}
\begin{proof}
First consider the half-plane geometry, without assuming that the
third inequality is valid for the actual color index $8$. With
$\delta=m_0-x$, the cap on $x$ and the three sloping inequalities are
\[
 \delta\ge\max\left\{
 0,\ \frac{(g+3)y-1}{56},\
 \frac{(2g+3)y-2}{42},\
 \frac{(32g+33)y-32}{420}\right\},\qquad 0\le y\le\frac1g.
\]
The successive slopes strictly increase: the differences between
the second and first nonconstant slopes and between the third
and second are $(5g+3)/168$ and $(16g+4)/560$. Their consecutive
crossing heights, including the zero line, are
\[
 \frac1{g+3}<\frac5{5g+3}<\frac{16}{16g+4}<\frac1g.
\]
Thus each line forms the upper envelope between its consecutive
crossings. The last deficit is $11/(140g)<m_0$ for $g\ge39$,
so no further intercept with $x=0$ occurs before the high cap.
Substitution gives precisely \eqref{S-eq:cr-seven-vertices}, in
boundary order, and proves that their convex hull is $\Gamma_i$.

We next describe the finite construction certificates and why they
imply all parameters. A scaffold is a cyclic word over
$\{0,1,\ldots,5,*\}$, where a mark $*$ uses a low-five vacancy edge.
Its low projection is a closed walk of the complete low-five model.
For each consecutive pair of marks let $r_j$ be its raw gap. The
certificate supplies an offset $\epsilon_j\in\mathcal E$ and a
designated cut after the left mark and before the edge emitting the
right mark. At that cut a word $V_j$ over $\{0,\ldots,5\}$ has
length $84$, contains $65$ low-five occurrences, is legal, and
returns to the same complete low-four mask and capped age-$5$ state.
The construction check verifies these equalities on
every supplied loop and verifies the scaffold's original transitions.
All scaffold lengths exceed the model's $31$-position memory, so
their closed states agree with their own periodic low past.

For a base index $i_0$, the integers
\[
 q_j=\frac{6i_0+\epsilon_j-r_j}{84}
\]
are checked nonnegative. At $i=i_0+14t$, insert $q_j+t$ copies
of $V_j$ at that cut and replace the marks by color $i$.
Closed low-state returns preserve all low packing constraints,
including boundaries between pieces. The high gaps become exactly
$6i+\epsilon_j$. They are legal original gaps, and the sum of two
exceeds the largest forbidden displacement $6i$. Thus older high
occurrences cannot conflict. This proves the infinite periodic
legality of the constructed word for every integer $t\ge0$.
If the scaffold has length $R$, low count $M$ and $k$ marks,
the resulting counts are
\begin{equation}\label{S-eq:cr-loop-counts}
 L_t=R+84\sum_jq_j+84kt,\qquad
 M_t=M+65\sum_jq_j+65kt,\qquad n_i=k.
\end{equation}
In particular the integer deficit $65L_t-84M_t$ is unchanged.

The complete numerical parameters of the four families are:
\begin{center}\small
\begin{tabular}{@{}crrrrrrr@{}}\toprule
Vertex&$i_0$&$R$&raw low count&$k$&$\sum q_j$&$L_0$&$M_0$\\\midrule
$E$&8&84&65&2&0&84&65\\
$F$&8&198&153&5&0&198&153\\
$G$&36&1048&810&16&27&3316&2565\\
$H$&36&894&691&10&14&2070&1601\\\bottomrule
\end{tabular}
\end{center}
Here is the full raw gap and padding inventory, with multiplicities:
\[
 \begin{array}{c|l}
 & (r_j,\epsilon_j,q_j;\ \text{multiplicity})\\\hline
 E&(40,-8,0;1),\ (44,-4,0;1)\\
 F&(39,-9,0;2),\ (40,-8,0;3)\\
 G&(39,-9,2;8),\ (40,-8,2;4),\
          (123,-9,1;3),\ (207,-9,0;1)\\
 H&(39,-9,2;6),\ (123,-9,1;2),\ (207,-9,0;2).
 \end{array}
\]
The order of the low symbols and the location of the cuts determine
whether each inserted loop returns to the required history. The $F,G,H$ base words also pass respectively
$2436$, $40416$ and $25272$ direct signed-distance comparisons.

For $E$ the scaffold is the already printed $84$-word $W_*$
in Section~\ref*{M-sec:six-labels}, with its two $8$'s marked. Its
gaps are $40,44$. Erasing those marks gives a periodic low word
with $65$ occupied positions. Immediately after either mark,
take the cyclic rotation of that erased $84$-word as the loop.
It returns to the same finite past and age state. This gives the
$E$ row directly, including all $i=8+14t$.

For $F$ the length-198 scaffold is the following word; its five 8's mark
the vacancies to be filled by color $i$.
\begin{center}\small
\begin{tabular}{@{}l@{}}
\texttt{15141801310120121341010102131251014130121012813101}\\
\texttt{45121312010101041213125101013012181241310105121312}\\
\texttt{01410103121012513140101821312100141351210120131014}\\
\texttt{012131251010184121312010101351210124131010012132}\\
\end{tabular}
\end{center}
With positions numbered from zero, its marks are $5,44,83,123,163$.
The period $W_8$ in Section~\ref*{M-sec:six-labels} uses a different
ordering. In \path{packing-closed-density-results.zip}, read the scaffold from
\begin{center}\small
\path{data/eventual_residue8_refinement/ratio_1_5/RESULT.json}.
\end{center}
The field \path{counterexample.labels} uses 99 for a mark.
The matching cuts and loops are in the $F$ entry of
\begin{center}\small
\path{runs/p26_full_initial/VERTEX_FAMILIES_REPLAY.json}.
\end{center}
Read the five records in \path{selected_cuts}; their zero-based cut
positions are $31,45,84,137,164$.
For $i=8+14t$, insert $t$ copies of each listed \path{loop_word}
immediately before the zero-based position \path{cut}, then replace
the marks by $i$.

For the rows $F,G,H$, the finite certificate tests every mark cut
and the associated $84$-step return. There are five, sixteen and
ten selected loops, respectively. Formula~\eqref{S-eq:cr-loop-counts}
now gives
\[
 \begin{array}{c|ccc|c}
 &L_t&M_t&n_i&65L_t-84M_t\\\hline
 E&84+168t&65+130t&2&0\\
 F&198+420t&153+325t&5&18\\
 G&3316+1344t&2565+1040t&16&80\\
 H&2070+840t&1601+650t&10&66.
 \end{array}
\]
The parameter in each row is $i=i_0+14t$ with its stated $i_0$.
Thus $L_E=2(g+3)$, $L_F=5g+3$, $L_G=16g+4$, and $L_H=10g$;
using $M_t/L_t=m_0-(65L_t-84M_t)/(84L_t)$ gives exactly the four
points in \eqref{S-eq:cr-seven-vertices}. The all-vacancy word,
an optimal low-five word, and the high-only word with gap $g$
give $O,A_0,Z_i$. The last word is legal because $g$ is an
allowed gap and $2g>6i$.

Fix $i\ge36$ in the residue and a convex combination of these
seven vectors, with coefficients $\theta_j\ge0$ summing to one.
If the period of vertex word $j$ is $P_j$, stage $n$ uses
$\lfloor\theta_j n/P_j\rfloor$ copies of it, grouping each word's
copies together. Separate nonempty groups and stages by a fixed
number of vacancies greater than $6i$. This removes every
cross-group equal-color conflict. A stage has useful length
$n+O_i(1)$ and its counts equal $n$ times the target vector
plus $O_i(1)$. There are at most seven buffers per stage.
After $n$ stages the length is $n(n+1)/2+O_i(n)$, with total
rounding and buffer error $O_i(n)$. A partial final stage has
length $O_i(n)$ and is negligible. Hence both ordinary prefix
densities converge at every position. Reflect on the negative
half-line and use a fixed buffer at the origin for a two-sided
coloring. This proves attainment of the entire real convex hull.
\end{proof}

\subsection{Validity of the first boundary segment from color eight}
\label{S-supp:cr-first-face}

\begin{Slemma}\label{S-lem:cr-first-face}
For every $i=8+14t$, $t\ge0$, put $C_i=6i-6$ and $L_i=30i-42$.
The entire exposed face of the ordinary region on
\[
 56x+C_i y=\frac{133}{3}
\]
is the segment $E_iF_i$ from \eqref{S-eq:cr-seven-vertices}. Its
endpoints are periodically attained and every real point is
ordinarily attained. The least coefficient in
\[
 n_i(I)\le\frac{|I|}{C_i}
       +A_i\left(\frac{65|I|}{84}-M_5(I)\right)+O_i(1)
\]
is $A_i=56/C_i$.
\end{Slemma}
\begin{proof}
The first two bounds of Lemma~\ref{S-lem:cr-residue8-supports}
hold at every stated index. On the first line, the cap
$x\le m_0$ implies $C_i y\ge1$. Substitute
$84x=133/2-(3C_i/2)y$ into the second bound to obtain
$(15i-21)y\le5/2$, or $y\le5/L_i$. Thus every actual limit
point on the line lies between $E_i$ and $F_i$. Their periodic
families are valid from $8$, and the preceding growing-block
construction with two vertex words realizes the segment.

The first finite bound, divided by $3C_i$, gives the stated
inequality with $A_i=56/C_i$ and a fixed boundary constant.
At $F_i$ the positive low deficit is $3/(14L_i)$ and the
high gain above $1/C_i$ is $12/(C_iL_i)$, since $5C_i-L_i=12$.
Their ratio is $56/C_i$. Repeating that periodic word rules
out every smaller coefficient whatever its fixed boundary
constant. This determines this face and coefficient, without
asserting the rest of the region at $i=8$.
\end{proof}

\subsection{Excluding the predicted corner at color twenty-two}
\label{S-supp:cr-onset}

\begin{Slemma}\label{S-lem:cr-exclude22}
The point $(317/410,1/123)$ is not a limit of simultaneous
low-five/color-$22$ densities of legal finite intervals.
\end{Slemma}
\begin{proof}
Let $Q$ be the checked $K=123$, $(a,b,c)=(5,32,-255)$ prefix
potential. On the full actual color-$22$ age model, use exact
ages $1,\ldots,132$ and saturated age $133$. The bounded
potential constructed in Lemma~\ref{S-lem:cr-prefix-clock} is
\[
 \Phi(p,d)=Q(p,j(d))-32d\quad(d<133),\qquad
 \Phi(p,133)=\max_{133\le D\le216}\{Q(p,j(D))-32D\}.
\]
For each edge define two nonnegative slacks,
\begin{align*}
 \sigma_s&=\Delta\Phi-(420\ell+3969h-357),\\
 \sigma_g&=\Delta V-(123h-1),\qquad
              V(d)=123-\min(d,123),
\end{align*}
where $h$ indicates a color-$22$ mark. The first nonnegativity
is the prefix certificate. The second follows directly from
age updates: it equals one on an ordinary edge with source
age at least $123$, and zero on every other edge. Every legal
high mark has source age at least $123$.

The additional finite certificate has low-five states and
actual ages $1,\ldots,123$. It includes every ordinary low
edge at an age below $123$, and at age $123$ only a high
mark on a genuine low-five vacancy, resetting the high age
to one. The color-$5$ age advances at this mark; it does not
reset or emit $5$. Retain precisely the edges on which the
raw $Q$ difference equals
\[
 5(84\ell-65)\quad\hbox{on an ordinary edge},\qquad
 -325-255+9\cdot32=-292\quad\hbox{on a high mark}.
\]
The clock agrees with the actual age on $1,\ldots,123$.
The verifier recomputes old tightness on the entire
candidate graph and verifies an integer potential $P$, with
$0\le P\le13$, satisfying
\begin{equation}\label{S-eq:cr-corner-P}
 P(q)-P(p)\ge h
\end{equation}
on every retained edge. Its complete inventory is
\begin{center}\small
\begin{tabular}{@{}lrr@{}}\toprule
Family&candidate edges&$Q$-tight edges\\\midrule
Ordinary low-four output or vacancy&2345490504&117676726\\
Color-$5$ reset&216808396&188213873\\
High mark at actual age $123$&7870094&855363\\\midrule
Total&2570168994&306745962\\\bottomrule
\end{tabular}
\end{center}
All old candidate slacks are nonnegative, and all retained
new slacks in \eqref{S-eq:cr-corner-P} are nonnegative. This is
an all-edge certificate on the specified actual-age graph,
not a test of one proposed periodic word.

Suppose intervals of lengths tending to infinity had the
claimed limiting point. Extract a convergent subsequence of
their edge-frequency vectors in the full finite actual-age
model. It limits to a circulation $f$ of total mass one,
with high-mark mass $1/123$. Boundedness of both potentials
and the equality coordinates give
\[
 \sum_e f(e)\sigma_s(e)
       =357-420\frac{317}{410}-3969\frac1{123}=0,
 \qquad
 \sum_e f(e)\sigma_g(e)=1-123\frac1{123}=0.
\]
Nonnegativity removes all positive-slack edges from its
support. In particular all ordinary edges at ages at least
$123$ have zero flow. Their incoming chains are the only
way to reach ages $124,\ldots,133$, so flow conservation
removes those vertices too, including the saturated self
edge. Thus the surviving flow is on the actual gap-$123$
model just described.

On ordinary edges there, $\Delta\Phi=\Delta Q-32$; on its
marks, $\Delta\Phi=\Delta Q+32(123-1)$. Hence
$\sigma_s=0$ is exactly the tested raw $Q$-tightness in
both cases. We may sum \eqref{S-eq:cr-corner-P} against the
remaining circulation. The potential differences cancel,
giving $0\ge1/123$, a contradiction.

The argument concerns the support of a limiting circulation.
It does not claim that every gap in an actual upper-density
maximizer is $123$. It excludes the point even when individual
intervals have transient or exceptional gaps.
\end{proof}

For completeness, the geometric calculation in
Lemma~\ref{S-lem:cr-residue8-vertices} puts
$(317/410,1/123)=H_{22}$ in $\Gamma_{22}$. All three actual
supports are valid at $22$, so
$\mathcal R_{22}\subsetneq\Gamma_{22}$. Every later index
$36+14t$ has the upper inclusion and all seven lower
constructions. This proves the exact tail-start assertion
of Theorem~\ref*{M-thm:cr-residue8}: starts $8$ and $22$ would
include the failed index $22$, whereas start $36$ succeeds.
The unrestricted integer cutoff is consequently $23$.

\subsection{Which certificates support the stated bounds}
\label{S-supp:cr-proof-roles}

The required added data have two roles. The seven-color
module contains the two complete $37$-layer actual-age
potentials, their outgoing verifiers, the $1084$ equality
word, and the five component color potentials. It retains
the complete low-five critical-graph dependency; the $6640$
small identities alone do not certify completeness. The
already proved six-label totals and weighted color-$8$
support supply the inherited endpoint bounds used in the
main proof. The printed $252$ corner is reused.

The residue-eight module contains the three complete prefix
potentials, the actual-$22$ tight-edge potential, and the
$E,F,G,H$ scaffold/cut/return-word data. The finite original
transitions, all designated return equalities, low counts,
gap congruences and literal metric checks establish their
parameter formulas as proved above. The prefix proof and
the actual-$22$ circulation argument give the upper scope
and the least start.

Verification checks all selected transitions and integer identities.
The finite-model interpretation, bounded-potential transfer and limit
arguments above remain necessary parts of the proof.

\section{Short equality witnesses}\label{S-supp:short-witnesses}

\subsection{The sharp low-four/color-five endpoint}
For a directly inspectable sharpness witness, concatenate the following five
rows and repeat. The symbol $0$ denotes a vacancy; its counts for colors
$1,2,3,4,5$ are $106,35,25,16,11$.
\begin{center}\footnotesize\ttfamily
21012103151012413100121021310145101312012101041013\\
12512100131421010130121512413101001213120141001315\\
21012134101010213121501413102101210310141521312100\\
10101421312105101312410101001213125141013012101201\\
31410512132101010413121021013105141012312101001314
\end{center}
The displayed word is generated from the certificate and checked against
every signed displacement forbidden for its color in the infinite graph.

\subsection{Colors 1--5 with one additional color}
For completeness, the words used above are printed below; concatenate
successive rows with the same name and repeat periodically. The label $0$
means an unoccupied position. Each $W_i$ uses colors $1$--$5$ and $i$;
$W_*$ is the low-optimal color-$8$ endpoint.
\begin{center}\small
\begin{tabular}{rl}
$W_{6}$ & \texttt{610101021312154101310210121061413152101213010141021312156101} \\
 & \texttt{314210121031010152131214} \\
$W_{7}$ & \texttt{710131021012143101015213121041010172131215014131021012103101} \\
 & \texttt{415213121001017142131215010131021012143101015213121041710102} \\
 & \texttt{131215014131021012103101415213121701010142131215010131021012} \\
 & \texttt{143101715213121041010102131215014131021712103101415213121001} \\
 & \texttt{010142131215} \\
$W_{8}$ & \texttt{031218210131401510123121010413102101251310148121312010101041} \\
 & \texttt{213125101013012102141310815121321014100131210215131401010123} \\
 & \texttt{121081413512101201310140121312510101041213128101013512101241} \\
 & \texttt{310100121312514101} \\
$W_{*}$ & \texttt{031418152131210010141021312150101314210121031810152131214010} \\
 & \texttt{101021312154101310210121} \\
\end{tabular}
\end{center}

\subsection{The simultaneous 6,7,8 corner}
Concatenate these rows and repeat to obtain the corner word:
\begin{center}\small
\begin{tabular}{l}
\texttt{413182171210316141521312100101014213121571613102181214310101} \\
\texttt{521312104161017213121501413182101210316141521312100101714213} \\
\texttt{121501613102181214010131521712134161010213121501413182101210} \\
\texttt{316141521312170101014213121501613102181214310171521312104161} \\
\texttt{010213121501} \\
\end{tabular}
\end{center}

\section{The complete periods}\label{S-supp:periods}

Read left to right and concatenate rows; the first entry has residue $0$.

\subsection*{\texorpdfstring{$\D{6}$}{D(1,6)}, colors 1 through 20}
%% Generated by build_d6_construction.py from D6_CURRENT_CONSTRUCTION.json.
Work on positions $0,\ldots,2519$. For each row of the first table,
place the indicated color at every residue in the specified period.
\begin{center}
\begin{tabular}{rrl}
\toprule
Color & Period & Residues \\
\midrule
1 & 7 & $\{1,3,6\}$ \\
2 & 14 & $\{5,9\}$ \\
4 & 84 & $\{0,16,32,53,68\}$ \\
5 & 21 & $\{18\}$ \\
6 & 28 & $\{14\}$ \\
7 & 252 & $\{7,40,74,114,147,193,226\}$ \\
8 & 84 & $\{4,49\}$ \\
\bottomrule
\end{tabular}
\end{center}
For color $3$, use the two patterns
\begin{align*}
A_0&=\{2,12,21,35,44,54,63,77\},\\
A_1&=\{2,12,21,35,44,58,67,77\}.
\end{align*}
In the consecutive blocks $84b+[0,83]$, place color $3$ at $84b+A_{\epsilon_b}$,
where the binary choices, in order, are
\[
\texttt{010010110010010010010010110010}.
\]
All these prescribed positions are disjoint. In each block
$252b+[0,251]$, exactly $35$ positions are still empty. List them in
increasing order as $r_{b,0},\ldots,r_{b,34}$, and give $r_{b,j}$ the
color in row $b$, column $j$ of the following two panels.
\begin{center}\small
\begin{tabular}{r|rrrrrrrrrrrrrrrrrr}
\toprule
$b\backslash j$ & 0 & 1 & 2 & 3 & 4 & 5 & 6 & 7 & 8 & 9 & 10 & 11 & 12 & 13 & 14 & 15 & 16 & 17 \\
\midrule
0 & 9 & 19 & 10 & 15 & 16 & 11 & 9 & 13 & 12 & 17 & 10 & 14 & 18 & 15 & 11 & 9 & 16 & 12 \\
1 & 10 & 17 & 9 & 12 & 15 & 11 & 18 & 13 & 10 & 16 & 14 & 12 & 9 & 19 & 11 & 15 & 10 & 17 \\
2 & 9 & 18 & 13 & 15 & 16 & 10 & 12 & 9 & 14 & 17 & 11 & 19 & 13 & 15 & 10 & 9 & 18 & 12 \\
3 & 9 & 11 & 12 & 15 & 17 & 10 & 19 & 14 & 9 & 13 & 11 & 16 & 12 & 10 & 15 & 9 & 18 & 17 \\
4 & 20 & 15 & 9 & 11 & 16 & 19 & 14 & 10 & 12 & 18 & 13 & 11 & 9 & 17 & 10 & 15 & 16 & 12 \\
5 & 19 & 11 & 18 & 15 & 10 & 12 & 17 & 14 & 9 & 13 & 11 & 10 & 16 & 12 & 15 & 9 & 19 & 18 \\
6 & 9 & 15 & 17 & 11 & 10 & 16 & 20 & 14 & 12 & 9 & 10 & 13 & 18 & 19 & 11 & 15 & 9 & 12 \\
7 & 10 & 12 & 17 & 9 & 15 & 11 & 18 & 13 & 10 & 14 & 19 & 16 & 9 & 11 & 12 & 15 & 10 & 17 \\
8 & 9 & 18 & 13 & 15 & 16 & 11 & 10 & 9 & 14 & 17 & 12 & 19 & 13 & 11 & 15 & 9 & 18 & 10 \\
9 & 9 & 10 & 17 & 15 & 11 & 12 & 19 & 14 & 9 & 13 & 10 & 16 & 11 & 15 & 12 & 9 & 17 & 18 \\
\bottomrule
\end{tabular}
\end{center}
\begin{center}\small
\begin{tabular}{r|rrrrrrrrrrrrrrrrr}
\toprule
$b\backslash j$ & 18 & 19 & 20 & 21 & 22 & 23 & 24 & 25 & 26 & 27 & 28 & 29 & 30 & 31 & 32 & 33 & 34 \\
\midrule
0 & 19 & 10 & 13 & 9 & 14 & 11 & 17 & 12 & 15 & 10 & 18 & 9 & 16 & 13 & 11 & 14 & 19 \\
1 & 13 & 9 & 18 & 14 & 12 & 10 & 11 & 16 & 9 & 15 & 13 & 19 & 17 & 12 & 14 & 10 & 11 \\
2 & 14 & 16 & 11 & 9 & 10 & 13 & 17 & 12 & 19 & 15 & 9 & 11 & 14 & 18 & 16 & 10 & 13 \\
3 & 14 & 11 & 13 & 9 & 10 & 12 & 19 & 16 & 15 & 11 & 9 & 14 & 18 & 10 & 17 & 13 & 12 \\
4 & 14 & 9 & 11 & 19 & 10 & 13 & 18 & 15 & 9 & 12 & 17 & 11 & 10 & 14 & 16 & 9 & 13 \\
5 & 14 & 11 & 10 & 9 & 13 & 17 & 12 & 16 & 15 & 11 & 9 & 10 & 14 & 19 & 18 & 12 & 13 \\
6 & 10 & 14 & 16 & 17 & 13 & 11 & 9 & 15 & 10 & 12 & 18 & 19 & 9 & 13 & 11 & 14 & 16 \\
7 & 13 & 9 & 18 & 14 & 11 & 10 & 12 & 16 & 9 & 15 & 13 & 19 & 17 & 11 & 14 & 12 & 10 \\
8 & 14 & 16 & 12 & 9 & 11 & 13 & 17 & 10 & 19 & 15 & 9 & 14 & 11 & 12 & 18 & 16 & 13 \\
9 & 14 & 10 & 13 & 9 & 19 & 11 & 12 & 16 & 10 & 15 & 9 & 14 & 17 & 13 & 11 & 18 & 12 \\
\bottomrule
\end{tabular}
\end{center}
This specifies the full word without vacancies. Repeat it with period
$2520$ on $\mathbb Z$; the largest label is $20$.

\clearpage
\subsection*{\texorpdfstring{$\D{8}$}{D(1,8)}, colors 1 through 22}

\begin{center}\footnotesize
\setlength{\tabcolsep}{3pt}
\begin{tabular}{rrrrrrrrrrrrrrrrrr}
 1 &  4 &  1 &  2 &  1 &  7 &  1 &  2 &  3 &  1 &  8 &  1 &  3 &  1 &  4 &  1 &  5 &  2 \\
 1 & 11 &  1 &  2 &  1 & 13 &  1 & 20 &  3 &  1 &  6 &  1 &  3 &  1 &  2 &  1 &  4 &  2 \\
 1 &  5 &  1 &  9 &  1 &  7 &  1 & 10 &  3 &  1 &  2 &  1 &  3 &  1 &  2 &  1 & 16 &  4 \\
 1 &  8 &  1 &  5 &  1 & 22 &  1 &  2 &  3 &  1 &  2 &  1 &  3 &  1 & 19 &  1 & 14 &  6 \\
 1 &  4 &  1 &  2 &  1 &  7 &  1 &  2 &  3 &  1 & 12 &  1 &  3 &  1 &  4 &  1 & 15 &  2 \\
 1 &  9 &  1 &  2 &  1 &  5 &  1 & 11 &  3 &  1 &  6 &  1 &  3 &  1 &  2 &  1 &  4 &  2 \\
 1 &  8 &  1 & 10 &  1 &  7 &  1 &  5 &  3 &  1 &  2 &  1 &  3 &  1 &  2 &  1 & 17 & 13 \\
 1 &  4 &  1 &  6 &  1 & 21 &  1 &  2 &  3 &  1 &  2 &  1 &  3 &  1 &  4 &  1 &  5 &  9 \\
 1 & 18 &  1 &  2 &  1 &  7 &  1 &  2 &  3 &  1 &  8 &  1 &  3 &  1 &  6 &  1 &  4 &  2 \\
 1 &  5 &  1 &  2 &  1 & 12 &  1 & 14 &  3 &  1 & 11 &  1 &  3 &  1 &  2 &  1 & 16 &  2 \\
 1 &  4 &  1 &  5 &  1 &  7 &  1 &  6 &  3 &  1 &  2 &  1 &  3 &  1 &  2 &  1 & 15 & 20 \\
 1 &  8 &  1 &  4 &  1 &  9 &  1 &  2 &  3 &  1 &  2 &  1 &  3 &  1 &  5 &  1 &  4 &  6 \\
 1 & 10 &  1 &  2 &  1 &  7 &  1 &  2 &  3 &  1 & 19 &  1 &  3 &  1 & 22 &  1 & 13 &  2 \\
 1 &  4 &  1 &  2 &  1 &  5 &  1 & 17 &  3 &  1 &  6 &  1 &  3 &  1 &  2 &  1 & 11 &  2 \\
 1 &  8 &  1 &  4 &  1 &  7 &  1 &  5 &  3 &  1 &  2 &  1 &  3 &  1 &  2 &  1 &  4 & 18 \\
 1 &  9 &  1 &  6 &  1 & 12 &  1 &  2 &  3 &  1 &  2 &  1 &  3 &  1 & 14 &  1 &  5 & 10 \\
 1 &  4 &  1 &  2 &  1 &  7 &  1 &  2 &  3 &  1 &  8 &  1 &  3 &  1 &  4 &  1 & 15 &  2 \\
 1 &  5 &  1 &  2 &  1 &  6 &  1 & 16 &  3 &  1 & 21 &  1 &  3 &  1 &  2 &  1 &  4 &  2 \\
 1 & 11 &  1 &  5 &  1 &  7 &  1 &  9 &  3 &  1 &  2 &  1 &  3 &  1 &  2 &  1 &  6 &  4 \\
 1 &  8 &  1 & 13 &  1 & 20 &  1 &  2 &  3 &  1 &  2 &  1 &  3 &  1 &  5 &  1 & 17 & 19 \\
 1 &  4 &  1 &  2 &  1 &  7 &  1 &  2 &  3 &  1 & 10 &  1 &  3 &  1 &  4 &  1 &  5 &  2 \\
 1 & 12 &  1 &  2 &  1 & 14 &  1 &  6 &  3 &  1 &  8 &  1 &  3 &  1 &  2 &  1 &  4 &  2 \\
 1 &  5 &  1 &  9 &  1 &  7 &  1 & 11 &  3 &  1 &  2 &  1 &  3 &  1 &  2 &  1 & 18 &  6 \\
 1 &  4 &  1 &  5 &  1 & 15 &  1 &  2 &  3 &  1 &  2 &  1 &  3 &  1 &  4 &  1 & 16 & 10 \\
 1 &  8 &  1 &  2 &  1 &  7 &  1 &  2 &  3 &  1 &  6 &  1 &  3 &  1 &  5 &  1 &  4 &  2 \\
 1 &  9 &  1 &  2 &  1 & 12 &  1 & 22 &  3 &  1 & 13 &  1 &  3 &  1 &  2 &  1 &  5 &  2 \\
 1 &  4 &  1 &  6 &  1 &  7 &  1 & 14 &  3 &  1 &  2 &  1 &  3 &  1 &  2 &  1 & 17 &  8 \\
 1 &  5 &  1 &  4 &  1 & 10 &  1 &  2 &  3 &  1 &  2 &  1 &  3 &  1 &  6 &  1 &  4 & 11 \\
 1 & 19 &  1 &  2 &  1 &  7 &  1 &  2 &  3 &  1 &  5 &  1 &  3 &  1 & 21 &  1 &  9 &  2 \\
 1 &  4 &  1 &  2 &  1 & 15 &  1 &  6 &  3 &  1 & 12 &  1 &  3 &  1 &  2 &  1 &  8 &  2 \\
 1 &  5 &  1 &  4 &  1 &  7 &  1 & 16 &  3 &  1 &  2 &  1 &  3 &  1 &  2 &  1 &  4 &  6 \\
 1 & 10 &  1 &  5 &  1 & 18 &  1 &  2 &  3 &  1 &  2 &  1 &  3 &  1 &  9 &  1 & 13 & 20 \\
 1 &  4 &  1 &  2 &  1 &  7 &  1 &  2 &  3 &  1 &  6 &  1 &  3 &  1 &  4 &  1 & 14 &  2 \\
 1 & 11 &  1 &  2 &  1 &  5 &  1 & 17 &  3 &  1 &  8 &  1 &  3 &  1 &  2 &  1 &  4 &  2 \\
 1 & 22 &  1 &  6 &  1 &  7 &  1 &  5 &  3 &  1 &  2 &  1 &  3 &  1 &  2 &  1 & 10 &  4 \\
 1 & 12 &  1 &  9 &  1 & 15 &  1 &  2 &  3 &  1 &  2 &  1 &  3 &  1 &  6 &  1 &  5 & 19 \\
 1 &  4 &  1 &  2 &  1 &  7 &  1 &  2 &  3 &  1 & 13 &  1 &  3 &  1 &  4 &  1 & 16 &  2 \\
 1 &  5 &  1 &  2 &  1 & 11 &  1 &  6 &  3 &  1 & 21 &  1 &  3 &  1 &  2 &  1 &  4 &  2 \\
 1 &  8 &  1 &  5 &  1 &  7 &  1 & 14 &  3 &  1 &  2 &  1 &  3 &  1 &  2 &  1 &  9 &  4 \\
 1 & 10 &  1 & 18 &  1 & 12 &  1 &  2 &  3 &  1 &  2 &  1 &  3 &  1 &  5 &  1 & 17 &  6 \\
\end{tabular}
\end{center}

\clearpage
\subsection*{\texorpdfstring{$\D{9}$}{D(1,9)}, colors 1 through 17}

\begin{center}\footnotesize
\setlength{\tabcolsep}{3pt}
\begin{tabular}{rrrrrrrrrrrrrrrrrr}
 1 &  3 &  1 &  2 &  1 &  3 &  1 &  2 &  1 &  5 &  1 & 15 &  1 &  4 &  1 & 13 &  1 &  7 \\
 1 &  2 &  1 &  3 &  1 &  2 &  1 &  3 &  1 &  4 &  1 &  9 &  1 &  6 &  1 &  5 &  1 &  2 \\
 1 &  3 &  1 &  2 &  1 &  3 &  1 &  8 &  1 & 12 &  1 &  5 &  1 &  7 &  1 &  2 &  1 &  3 \\
 1 &  2 &  1 &  3 &  1 & 11 &  1 & 14 &  1 &  6 &  1 &  4 &  1 &  2 &  1 &  3 &  1 &  2 \\
 1 &  3 &  1 & 10 &  1 &  5 &  1 &  4 &  1 &  7 &  1 &  2 &  1 &  3 &  1 &  2 &  1 &  3 \\
 1 &  5 &  1 &  8 &  1 &  6 &  1 &  9 &  1 & 17 &  1 &  3 &  1 &  2 &  1 &  3 &  1 &  2 \\
 1 & 13 &  1 &  4 &  1 &  7 &  1 &  5 &  1 &  3 &  1 &  2 &  1 &  3 &  1 &  2 &  1 &  4 \\
 1 &  6 &  1 &  5 &  1 & 12 &  1 & 15 &  1 &  2 &  1 &  3 &  1 &  2 &  1 &  3 &  1 &  8 \\
 1 &  7 &  1 &  9 &  1 &  4 &  1 &  2 &  1 &  3 &  1 &  2 &  1 &  3 &  1 & 10 &  1 &  5 \\
 1 &  4 &  1 & 14 &  1 &  6 &  1 &  3 &  1 &  2 &  1 &  3 &  1 &  2 &  1 &  7 &  1 & 11 \\
 1 & 16 &  1 &  5 &  1 &  3 &  1 &  2 &  1 &  3 &  1 &  2 &  1 &  8 &  1 & 13 &  1 &  4 \\
 1 &  6 &  1 &  3 &  1 &  2 &  1 &  3 &  1 &  2 &  1 &  7 &  1 &  4 &  1 &  9 &  1 &  5 \\
 1 &  3 &  1 &  2 &  1 &  3 &  1 &  2 &  1 & 12 &  1 & 10 &  1 &  5 &  1 &  6 &  1 &  3 \\
 1 &  2 &  1 &  3 &  1 &  2 &  1 &  7 &  1 &  8 &  1 &  4 &  1 & 15 &  1 &  3 &  1 &  2 \\
 1 &  3 &  1 &  2 &  1 & 11 &  1 & 17 &  1 &  5 &  1 &  6 &  1 &  3 &  1 &  2 &  1 &  3 \\
 1 &  2 &  1 &  7 &  1 &  5 &  1 &  4 &  1 &  9 &  1 &  3 &  1 &  2 &  1 &  3 &  1 &  2 \\
 1 & 13 &  1 &  4 &  1 &  8 &  1 &  6 &  1 &  3 &  1 &  2 &  1 &  3 &  1 &  2 &  1 &  7 \\
 1 & 14 &  1 &  5 &  1 & 12 &  1 &  3 &  1 &  2 &  1 &  3 &  1 &  2 &  1 &  4 &  1 &  5 \\
 1 & 11 &  1 &  6 &  1 &  3 &  1 &  2 &  1 &  3 &  1 &  2 &  1 &  7 &  1 &  9 &  1 & 10 \\
 1 &  4 &  1 &  3 &  1 &  2 &  1 &  3 &  1 &  2 &  1 &  8 &  1 &  5 &  1 & 16 &  1 &  6 \\
\end{tabular}
\end{center}

\section{Minimum infinite distances}\label{S-supp:minima}

\begin{center}
\begin{tabular}{rrrr}
\toprule
Color $i$ & $\D{6}$ & $\D{8}$ & $\D{9}$ \\
\midrule
1 & 2 & 2 & 2 \\
2 & 3 & 3 & 4 \\
3 & 4 & 4 & 4 \\
4 & 5 & 5 & 6 \\
5 & 6 & 6 & 6 \\
6 & 7 & 7 & 8 \\
7 & 8 & 8 & 8 \\
8 & 9 & 9 & 10 \\
9 & 10 & 10 & 10 \\
10 & 11 & 11 & 12 \\
11 & 12 & 12 & 12 \\
12 & 13 & 13 & 14 \\
13 & 14 & 14 & 14 \\
14 & 15 & 15 & 16 \\
15 & 16 & 17 & 16 \\
16 & 17 & 17 & 24 \\
17 & 18 & 18 & 20 \\
18 & 19 & 19 & \text{--} \\
19 & 20 & 20 & \text{--} \\
20 & 94 & 21 & \text{--} \\
21 & \text{--} & 22 & \text{--} \\
22 & \text{--} & 23 & \text{--} \\
\bottomrule
\end{tabular}
\end{center}

\section{Limits of a printed density relaxation}\label{S-app:limits}

The density bounds are necessary conditions for a packing coloring; a
feasible vector for them need not be a coloring. We isolate a small
relaxation using only inequalities printed in this paper.

\subsection*{A small relaxation for seventeen colors}
Write $x=x_1+\cdots+x_5$ and maximize
$x+x_6+\cdots+x_{17}$ subject to exactly the following inequalities:
\[
 \begin{gathered}
 0\le x\le65/84,\qquad x_j\ge0\quad(6\le j\le17),\\
 548x+708(x_6+x_7)\le469,\qquad
 147x+220x_9\le355/3,\\
 x_j\le\frac1{6j-9}\qquad(8\le j\le17).
 \end{gathered}
\]
The first cap is the low-five optimum, the weighted seven-color row is
Theorem~\ref*{M-thm:cr-seven}, and the color-$9$ row is the limiting form
of Theorem~\ref*{M-thm:sharp-high-tangents}. The remaining rows are spacing
bounds and nonnegativity. No unprinted color-$7$ or joint-fiber constraint
is part of this diagnostic system. Its exact optimum is
\[
 U_{\rm small}=\frac{225935334894901}{225904042973700}>1.
\]
For an upper certificate, take $1/708$ of the weighted seven-color row,
$40/26019$ of the color-$9$ row, $17219/26019$ of the color-$9$ spacing
cap, and one copy of every other spacing cap for colors $8,\ldots,17$.
All weights are nonnegative. The coefficient on $x$ is
$(137+40)/177=1$, and the coefficient on $x_9$ is
$(8800+17219)/26019=1$; all the other objective coefficients are also one.
The right sides sum to $U_{\rm small}$. A matching feasible point is
\[
 (x,x_6,x_7)=\left(\frac{1021}{1323},\frac{261}{7154},\frac{279547}{9768276}\right),\qquad
 x_i=\frac1{6i-9}\quad(8\le i\le17).
\]
Direct substitution verifies every displayed row and gives objective
$U_{\rm small}$. Since the zero vector is feasible, scaling this point
by $1/U_{\rm small}$ gives a feasible point of total density one.
Thus these particular printed inequalities cannot exclude $17$ colors
by linear multipliers. This is not a $17$-coloring and does not limit
all possible joint-color methods. The diagnostic is not a premise of
Theorems~\ref*{M-thm:cr-seven} and~\ref*{M-thm:cr-residue8}.

\section{Reproducibility}\label{S-supp:replay}
The verification programs reconstruct finite graphs directly from the
coloring constraints. For the five initial certificates they regenerate
the history transitions, check all potential inequalities and attaining
words, and recompute the rational chromatic exclusions. Periodic
constructions are checked using infinite-graph distances. These checks
use separate implementations from the searches that produced the
potentials and colorings.

The structural results additionally require complete critical graphs,
component and color potentials, cycle lists, and age-product models.
The all-color capacity formula uses the Boolean mixing checks, phase
potentials and attaining cycles in the main proof. Joint-region
certificates use the same low-color history for simultaneously present
colors. Integer residual storage is lossless; the verifiers reconstruct
the potential values and use wide signed arithmetic. Their purpose is
to check every relevant transition, not to estimate a density from
sampled walks.

For component competition, the checks cover the twelve phase sets,
$10264$ component inequalities, the two $84$-period low words and
the four gap schedules. They verify the component walks, congruences,
infinite-graph distances and parameter-dependent period identities.
The logarithmic deficit theorem also uses the complete critical-layer
decomposition and the analytic run-length argument of
Section~\ref*{M-sec:component-competition}.

The sharp unrestricted tradeoffs use outgoing inequalities for the
complete joint models and exact equality words. The eventual-region
theorems further require the complete critical graph and the exact-length
argument of Section~\ref*{M-sec:eventual-high-regions}; checking individual
scalar inequalities alone does not establish the event-graph description.

The seven-color maximum and rigidity use the two complete three-age
certificates, the $252$- and $1084$-period witnesses and the color
identities on the fully identified optimal low-five graph
(Sections~\ref{S-supp:cr-certificates}--\ref{S-supp:cr-individual-rigidity}).
The residue-eight region uses the three prefix-clock certificates,
the actual-$22$ tight graph and the exact scaffold, cut and return-word
data (Sections~\ref{S-supp:cr-prefix-clock}--\ref{S-supp:cr-onset}).
The saturated-fiber transfer and ordinary-limit constructions in those
sections are necessary mathematical parts of the proof.

The data include the exact inputs, verifiers and instructions needed
for these checks. Computational resource requirements and execution
records are documented there. The rational relaxation in
Section~\ref{S-app:limits} is a separate diagnostic, not a premise of
the density theorems.

\end{document}